\documentclass[12pt,a4paper]{amsart}
\usepackage{amsfonts}
\usepackage{amssymb}
\usepackage{amsmath}
\usepackage{amsthm}
\usepackage{cite}
\usepackage{graphicx}
\usepackage{stmaryrd}
\usepackage{amscd}
\usepackage{color}
\newtheorem{theorem}{Theorem}
\theoremstyle{plain}

\newtheorem{corollary}[theorem]{Corollary}
\newtheorem{lemma}[theorem]{Lemma}

\numberwithin{equation}{section}

\def\ri{{\rm i}}

\newcommand{\R}{\mathbb{R}}
\newcommand{\C}{\mathbb{C}}

\newcommand{\N}{\mathbb{N}}

\newcommand{\Z}{\mathbb{Z}}

\newcommand{\cB}{{\mathcal B}}

\newcommand{\cD}{{\mathcal D}}
\newcommand{\cE}{{\mathcal E}}

\newcommand{\cH}{{\mathcal H}}

\newcommand{\cL}{{\mathcal L}}

\newcommand{\cP}{{\mathcal P}}

\renewcommand{\phi}{\varphi}

\newcommand{\bpm}{{\begin{pmatrix}}}
\renewcommand{\phi}{\varphi}
\newcommand{\pa}{{\partial}}
\def\barl#1{\overline{#1}}

\DeclareMathOperator{\curl}{curl}
\DeclareMathOperator{\const}{const.}

\DeclareMathOperator{\Id}{Id}

\DeclareMathOperator{\kernel}{Ker}
\DeclareMathOperator{\range}{Rg}

\DeclareMathOperator*{\essinf}{ess\,inf}
\DeclareMathOperator*{\esssup}{ess\,sup}

\newcommand{\dist}{\text{\rm dist}}
\newcommand{\supp}{\text{\rm supp}}

\def\blem{\begin{lemma}}\def\elem{\end{lemma}}
\def\bthm{\begin{theorem}}\def\ethm{\end{theorem}}
\def\bcor{\begin{corollary}}\def\ecor{\end{corollary}}
\def\beq{\begin{equation}}\def\eeq{\end{equation}}

\begin{document}

\title{Ground States of a Nonlinear Curl-Curl Problem in Cylindrically 
Symmetric Media}

\renewcommand\shorttitle{Ground States of a Nonlinear Curl-Curl Problem}

\author{Thomas Bartsch}
\address{Th. Bartsch \hfill\break 
Institut f\"ur Mathematik, University of Giessen\hfill\break
D-35392 Giessen, Germany}
\email{thomas.bartsch@math.uni-giessen.de}

\author{Tom\'{a}\v{s} Dohnal}
\address{T. Dohnal \hfill\break 
Fachbereich Mathematik, University of Dortmund\hfill\break
D-44221 Dortmund, Germany}
\email{tomas.dohnal@math.tu-dortmund.de}

\author{Michael Plum}
\address{M. Plum \hfill\break 
Institut f\"ur Analysis, Karlsruhe Institute of Technology (KIT)\hfill\break
D-76128 Karlsruhe, Germany}
\email{michael.plum@kit.edu}

\author{Wolfgang Reichel}
\address{W. Reichel \hfill\break 
Institut f\"ur Analysis, Karlsruhe Institute of Technology (KIT), \hfill\break
D-76128 Karlsruhe, Germany}
\email{wolfgang.reichel@kit.edu}

\date{\today}



\begin{abstract} We consider the nonlinear curl-curl problem $\nabla\times\nabla\times U + V(x) U= \Gamma(x)|U|^{p-1}U$ in $\R^3$ related to the nonlinear Maxwell equations for monochromatic fields. We search for solutions as minimizers (ground states) of the corresponding energy functional defined on subspaces (defocusing case) or natural constraints (focusing case) of $H(\curl;\R^3)$. Under a cylindrical symmetry assumption on the functions $V$ and $\Gamma$ the variational problem can be posed in a symmetric subspace of $H(\curl;\R^3)$. For a strongly defocusing case $\esssup \Gamma <0$ with large negative values of $\Gamma$ at infinity we obtain ground states by the direct minimization method. For the focusing case $\essinf \Gamma >0$ the concentration compactness principle produces ground states under the assumption that zero lies outside the spectrum of the linear operator $\nabla \times \nabla \times +V(x)$. Examples of cylindrically symmetric functions $V$ are provided for which this holds. 
\end{abstract}
\maketitle


\section{Introduction}

For given functions $V\in L^\infty(\R^3)$, $\Gamma\in L^\infty_{loc}(\R^3)\setminus\{0\}$ we consider the nonlinear curl-curl problem 
\beq\label{E:curl-curl-gen}
\nabla\times\nabla\times U + V(x) U= \Gamma(x)|U|^{p-1}U  \quad \mbox{ in } \R^3,
\eeq
where $p>1$, and look for real weak solutions
$$
U\in X := H(\curl;\R^3)\cap L_{|\Gamma|}^{p+1}(\R^3),
$$ 
where $L_{|\Gamma|}^{p+1}(\R^3)$ denotes the space of $L^{p+1}$-functions with respect to the measure $|\Gamma|\,dx$ and $H(\curl;\R^3)$ is the space of functions $U\in L^2(\R^3)$ for which $\curl U$ is defined in the sense of distributions and $\curl U\in L^2(\R^3)$; cf. Section~\ref{var_form} for more details on these spaces.
The solutions $U\in X$ of \eqref{E:curl-curl-gen} arise as critical points of the functional 
$$
J[U]= \int_{\R^3} \frac{1}{2}(|\nabla \times U|^2+ V(x) |U|^2) - \frac{\Gamma(x)}{p+1}|U|^{p+1}\,dx, \quad U\in X.
$$
We find ground state solutions, i.e. minimizers of $J$ within a certain subspace (defocusing case) or a natural constraint (focusing case) of $X$. Note that although we limit our attention to real solutions, the methods are in principle applicable in the complex case  $U(x)\in \C^3$ as well. 

\subsection{Variational aspects of the curl-curl problem}


In the literature there are only few results on the nonlinear curl-curl problem. In \cite{benci_fortunato_archive} Benci, Fortunato opened the discussion about ground states for the problem
\begin{equation}
\nabla\times \nabla \times U = W'(|U|^2)U.
\label{curlcurl_general}
\end{equation}
The problem was solved by Azzollini, Benci, D'Aprile, Fortunato in \cite{ABDF06} using variational and symmetry-based methods. Using a different class of symmetries D'Aprile, Siciliano also obtained in \cite{daprile_siciliano} solutions of \eqref{curlcurl_general}. Recently, Bartsch and Mederski \cite{bartsch_mederski} considered ground states as well as bound states of \eqref{curlcurl_general} on a bounded domain $\Omega$ with the boundary condition $\nu\times U=0$ on $\partial\Omega$. In \cite{Mederski_2014} Mederski considers \eqref{E:curl-curl-gen} where, e.g., the right hand side is of the form $\Gamma(x) f(u)$ with $f(u)\sim |u|^{p-1} u$ if $|u|\gg 1$ and $f(u)\sim |u|^{q-1}u$ if $|u|\ll 1$ for $1<p<5<q$ and where $\Gamma>0$ is periodic and bounded, $V\leq 0$, $V\in L^\frac{p+1}{p-1}(\R^3)\cap L^\frac{q+1}{q-1}(\R^3)$.

\medskip

Let us point out that on top of the common obstacle of $J$ being unbounded from below in case $\Gamma>0$, the variational formulation has the following additional difficulties:
\begin{itemize}
\item For $p>1$ the space $H(\curl;\R^3)$ does not embed into $L_{|\Gamma|}^{p+1}(\R^3)$ even when $\Gamma$ is bounded. In the so-called focusing case $\Gamma>0$ it is therefore hard to control the $X$-norm of any Palais-Smale sequence $(U_k)_{k\in \N}$, i.e., any sequence with $(J[U_k])_{k\in\N}$ bounded and $J'[U_k]\to 0$ as $k\to \infty$. 
\item Note that $\|\nabla U\|_2^2 = \|\nabla \times U\|_2^2+\|\nabla\cdot U\|_2^2$. Hence restriction of $J$ to the space $X_0=\{U\in X: \nabla\cdot
U=0\}$ on one hand allows at least for $\Gamma\in L^\infty(\R^3)$ the embedding $X_0\to L_{|\Gamma|}^{p+1}(\R^3)$ but on the other hand it generates an additional gradient term in the
Euler-Lagrange equation.
\end{itemize}
Therefore, finding critical points of $J$ directly in the whole space $X$ is out of the
scope of the current paper. Instead we will look for critical points on a
suitable subspace by exploiting symmetries of \eqref{E:curl-curl-gen}. As
proposed in \cite{ABDF06}, one such subspace is given by functions $U$ of the
form
\beq \label{symmetric_form}
U(x)=\frac{u(r,x_3)}{r}\begin{pmatrix} -x_2 \\ x_1\\ 0 \end{pmatrix}, \quad  r^2 =x_1^2+x_2^2,
\eeq
where $u: (0,\infty)\times \R\to \R$ is a real valued, scalar function. Assuming
cylindrical symmetry also for the potentials $V$ and $\Gamma$, i.e.,
$V=V(r,x_3), \Gamma=\Gamma(r,x_3)$, this ansatz leads to the equation
$$
\left(-\pa_r^2 -\pa_{x_3}^2-\frac{1}{r}\pa_r + \frac{1}{r^2}+V(r,x_3)\right)u=\Gamma(r,x_3)|u|^{p-1}u.
$$
We also define the linear operator in the vector valued equation \eqref{E:curl-curl-gen} as
\beq \label{def_Lvec}
\cL:=(\nabla\times\nabla\times) + V(x)
\eeq
and study its spectrum $\sigma(\cL)$ when restricted to a suitable subspace of functions which exhibit symmetries like the functions given in \eqref{symmetric_form}. Under the above symmetry assumptions, we will study \eqref{E:curl-curl-gen} in the following three scenarios:
\begin{itemize}
\item Fully radially symmetric case: $V=V(\rho)$, $\Gamma=\Gamma(\rho)$ with $\rho^2=r^2+x_3^2$.
\item Strongly defocusing case: $\esssup_{\R^3} V<0$ and  $\Gamma(x) \leq -C(1+|x|)^\alpha$ for some constants
$C>0$ and $\alpha>\frac{3(p-1)}{2}$.
\item Focusing case: $\essinf_{\R^3} \Gamma>0$, $0\not \in \sigma(\cL)$. Examples of such potentials $V(r,x_3)$ are given. They are periodic in the $x_3$-direction, satisfy $\lim_{r\to \infty} V(r,x_3)=V_\infty(x_3)$ and $\esssup_\R V_\infty(x_3) >0>\essinf_{\R^3} V$. Hence the potential has non-vanishing negative and, as $r\to \infty$, also non-vanishing positive part.
\end{itemize}
From a physical point of view the latter two scenarios may be criticized. Because $\Gamma$ corresponds to the electric susceptibility of the considered medium, see Section~\ref{S:physics}, the strongly defocusing case implies unrealistically high defocusing nature of the material. And since $V(x)$ is proportional to $-n^2(x)$, where $n$ is the refractive index, the condition of the non-vanishing positive part of $V$ at infinity in the focusing case implies an imaginary refractive index. Hence it will be desirable to overcome these limitations in future work.

\subsection{Main results}

Now we state our main results. The first result is concerned with those solutions of \eqref{E:curl-curl-gen} that are fully radially symmetric. 

\begin{theorem}[Fully radially symmetric case] \label{T:rad_sym}
Let $p>1$ and assume that $V,\Gamma\in L^\infty_{loc}(\R^3)$ and $0 \leq V \Gamma^{-1} \in L^\frac{p}{p-1}_{loc}(\R^3)$. Additionally suppose the full radial symmetry of $V$ and $\Gamma$ in $\R^3$, i.e. $V(x)=\tilde V(|x|)$ and $\Gamma(x)=\tilde\Gamma(|x|)$ for almost all $x\in \R^3$ and $\tilde V, \tilde \Gamma \in L^\infty_{loc}([0,\infty))$. Under the full radial symmetry condition $U(x)=M^T U(Mx)$ for a.a. $x\in \R^3$ and all $M\in O(3)$, all distributional solutions $U\in L^p_{loc}(\R^3)$ of \eqref{E:curl-curl-gen} satisfy $\nabla\times U=0$ and have the form 
\begin{align}\label{E:U_rad_sym}
U(x) = s(|x|) \left(\frac{V(x)}{\Gamma(x)}\right)^\frac{1}{p-1} \frac{x}{|x|},
\end{align}
where $s:(0,\infty)\to\{-1,1\}$ is an arbitrary measurable function. If additionally $(V\Gamma^{-1})^\frac{2}{p-1}$, $(V\Gamma^{-1})^\frac{p+1}{p-1}\Gamma\in L^1(\R^3)$ then $U\in X$ and hence it is a critical point of $J$. 
\end{theorem}

Thus, the assumption of full radial symmetry does not lead to interesting solutions of \eqref{E:curl-curl-gen}. We therefore relax the fully radial symmetry and look for solutions having only cylindrical symmetry. For this purpose we use in Theorem~\ref{defoc} and Theorem~\ref{foc} the space $X_{G_1}$
which will be defined in Section~\ref{var_form} and may be thought of as the subspace of $X=H(\curl;\R^3)\cap L_{|\Gamma|}^{p+1}(\R^3)$
consisting of vector fields of the form~\eqref{symmetric_form}. 

\begin{theorem}[Strongly defocusing case] Let $p>1$ and assume that $V=V(r,x_3)$ and $\Gamma=\Gamma(r,x_3)$ have cylindrical symmetry and satisfy
\begin{itemize}
\item[(i)] $\Gamma(x) \leq -C(1+|x|)^\alpha$ in $\R^3$ with $\alpha >
\frac{3}{2}(p-1)$ and $C>0$,
\item[(ii)] $V\in L^\infty(\R^3)$ and  $\esssup_{\R^3} V<0$. 
\end{itemize} 
Then \eqref{E:curl-curl-gen} has a ground state on $X_{G_1}$.
\label{defoc}
\end{theorem}

\begin{theorem}[Focusing case] Let $1<p<5$ and assume that $V=V(r,x_3)$ and $\Gamma=\Gamma(r,x_3)$ have cylindrical symmetry and satisfy 
\begin{itemize}
\item[(i)] $\essinf_{\R^3} \Gamma>0$,
\item[(ii)] $V,\Gamma \in L^\infty(\R^3)$ are $1$-periodic in the $x_3$-direction, i.e., $V(r,x_3)=V(r,x_3+1)$, $\Gamma(r,x_3)=\Gamma(r,x_3+1)$ for a.a. $r>0, x_3\in \R$,
\item[(iii)] $0\not \in \sigma(\cL)$.
\end{itemize}
Then \eqref{E:curl-curl-gen} has a ground state on $X_{G_1}$, which is moreover a minimizer of $J$ restricted to a natural constraint set (the so-called Nehari-Pankov manifold, cf. Section~\ref{S:foc}). 
\label{foc}
\end{theorem}

Examples of potentials $V(r,x_3)$ with $0\not\in \sigma(\cL)$ are given in Section~\ref{sec:spectrum}. They have non-vanishing positive and negative parts. 



\subsection{Physical context of the problem}\label{S:physics}

As we show below, equation \eqref{E:curl-curl-gen} is a generalization of the Kerr nonlinear Maxwell's equations in three dimensions for monochromatic waves when higher harmonics are neglected. Solutions $U\in H^1(\R^3)$ then correspond to fully localized standing electromagnetic waves. The problem of localizing light in all three dimensions attracts strong interest in the physics community. This is partly due to the potential applications of such ``light bullets'' as information carriers in future optical logic and optical computing devices. Standing bullets, in particular, can be used in optical memory.

So far, to our knowledge, standing light bullets have not been observed in experiments: neither in homogeneous or periodic media nor in radial or cylindrical geometries, corresponding to the choice in this paper. Nevertheless, at least one theoretical prediction of such waves exists. In the Kerr nonlinear fiber Bragg grating (a cylindrical geometry with periodicity in the longitudinal direction) an asymptotic model for broad wavepackets and a small periodicity contrast supports localized waves, so called gap solitons \cite{AW89}. The model is the system of one dimensional coupled mode equations and the gap solitons come in a family including standing solutions.  Gap solitons have been experimentally observed with velocities as low as $0.23 \tfrac{c}{n}$, where $c$ is the speed of light in vacuum and $n$ the average refractive index of the fiber core \cite{Moke-etal_06}, but not with velocity zero. On the other hand, moving localized pulses have been demonstrated in numerous other nonlinear geometries including standard optical fibers \cite{Mollenauer-etal_80} and arrays of waveguides arranged in the plane \cite{Minardi-etal_2010}. In most physics articles theoretical predictions of light bullets are made based on the nonlinear Schr\"odinger equation (NLS). For instance in homogenous materials the NLS is known to have radially symmetric localized solutions, so called Townes solitons, in all dimensions \cite{Sulem}. In periodic media \cite{Pankov05} and at interfaces of two periodic structures \cite{DPR11} standing ground state $H^1$-solutions exist. The NLS is, however, only an asymptotic approximation of Maxwell's equations. Moreover, for inhomogeneous media in two and three dimensions the approximation has not been rigorously justified. This paper, in contrast, deals with the full three dimensional Maxwell problem.

The three dimensional Maxwell equations in the absence of charges and currents read
\begin{align*}
\nabla\times \cE + \partial_t \cB &=0, \qquad \nabla \cdot \cD =0, \\
\nabla\times \cH - \partial_t \cD &=0, \qquad \nabla \cdot \cB=0.
\end{align*}
Here $\cE,\cH: \R^4\to \R^3$ denote the electric and
magnetic field, respectively, and $\cD, \cB:\R^4\to \R^3$ denote the displacement field and
the magnetic induction, respectively. For the relation between the magnetic field $\cH$ and the magnetic induction $\cB$ we assume
$\cB=\mu_0\cH$ with $\mu_0$ constant. By taking the curl of the first equation one finds 
\begin{equation}
\nabla\times\nabla \times \cE + \mu_0\partial_t^2 \cD=0, \quad \nabla\cdot \cD=0.
\label{maxwell_ed}
\end{equation}
For a Kerr-type nonlinear medium the material law between the electric field $\cE$ and the displacement field $\cD$
is given by 
\begin{equation}
\cD = \epsilon_0 \left(n^2(x)\cE+\cP_{\text{NL}}(x,\cE)\right) \quad \mbox{ with } \quad \cP_{\text{NL}}(x,\cE)=
\chi^{(3)}(x) (\cE \cdot \cE) \cE,
\label{Kerr}
\end{equation}
where $n^2(x)=1+\chi^{(1)}(x)$ is the square of the linear refractive index and where $\cP_{\text{NL}}$ denotes the nonlinear part of the polarization. Note that in this section we use the notation $w \cdot z = w_1z_1 + w_2z_2+w_3z_3$ both for real and complex valued vectors $w,z \in\C^3$. The functions $\chi^{(1)}$ and $\chi^{(3)}$ denote the linear and cubic susceptibilities of the medium respectively. Although $\chi^{(3)}$ is generally a tensor, symmetries in the atomic structure of the material allow a reduction to a scalar, see \cite[Sec.~2d]{MN04}. The resulting second order equation for the electric field $\cE$ is then given by the quasilinear wave equation
\begin{equation}
\nabla\times\nabla\times \cE + \frac{1}{c^2}\partial_t^2\bigl(n(x)^2
\cE+\chi^{(3)}(x)(\cE\cdot\cE)\cE\bigr)=0, \qquad (x,t)\in \R^4
\label{quasilinear}
\end{equation}
together with $\nabla\cdot\cD=0$. Here $c=(\epsilon_0\mu_0)^{-1/2}$ is the speed of light in vacuum. If $\cE$ solves \eqref{quasilinear}, then $\cD$ is known from \eqref{Kerr} and $\cB$ can be obtained from $\nabla\times \cE$ by a time integration and thus also $\cH$ is known. Moreover, the fields $\cD, \cB$ will be divergence free provided they are divergence free at some fixed time, e.g., $t=0$. 

The question of light bullets is that of the existence of
solutions of Maxwell's equations in nonlinear dispersive media which are
localized in space, i.e., which at all times $t$ are decaying to $0$ as $|x|\to
\infty$.

In this paper we cannot give a complete answer to this question. Instead we will solve a
related problem. Motivated by Fourier-expansion in time, one might look for a solution of \eqref{quasilinear} of the form $\cE(x,t)=\sum_{k=0}^\infty \left(e^{-\ri(2k+1)\omega
t}E_k(x)+\text{c.c.}\right)$ with $E_k(x)\in \C^3$. If such solution existed under the
additional assumption of localization, i.e., $\cE(x,t)$ decaying to $0$ as $|x|\to
\infty$ for all $t$, then it would be a standing light bullet. Here we consider the simpler monochromatic ansatz 
\begin{equation}
\cE(x,t)=e^{-\ri\omega t}E(x)+\text{c.c.} \mbox{ with } E(x)\in \C^3.
\label{monochromatic}
\end{equation}
If we insert these monochromatic fields into the constitutive relation \eqref{Kerr} and neglect the generation of higher harmonics, i.e., we cancel all terms with factors $e^{\pm3\ri \omega t}$, then we obtain the new simplified constitutive relation
$\cD = \epsilon_0 \left(n^2(x)\cE+\cP_{\text{NL}}^{\text{(a)}}(x,\cE)\right)$ with 
\begin{equation}
\cP_{\text{NL}}^{\text{(a)}}(x,e^{-\ri\omega t}E+\text{c.c.})=\chi^{(3)}(x) e^{-\ri\omega t} \left(2|E|^2 E+(E\cdot E)\bar E\right)+ \text{c.c.}.
\label{Kerr_new}
\end{equation}
Note that here $E\cdot E=E_1^2+E_2^2+E_3^2\in \C$ whereas $|E|^2= |E_1|^2+|E_2|^2+|E_3|^2$ denotes the Hermitian inner product of $E$ with itself.
The second order elliptic equation for the $E$-field resulting from \eqref{maxwell_ed} is 
\beq\label{E:curl-curl}
\nabla \times \nabla \times E -\frac{\omega^2}{c^2}\bigl(n^2(x)E +\chi^{(3)}(x)\big(2|E|^2\,E+(E\cdot E)\,\barl{E}\big)=0,
\qquad x\in \R^3.
\eeq
Note that the divergence conditions $\nabla \cdot \cD =0$ is automatically satisfied due to the monochromatic ansatz and the curl-curl structure of the equation. 

Another model of the nonlinear polarization which effectively removes higher harmonics and results in equation \eqref{E:curl-curl} is given by time-averaging $\cE \cdot \cE$. In detail, for a $T-$periodic $\cE(x,t)\cdot \cE(x,t)$ one defines
\begin{equation}
\cP_{\text{NL}}^{\text{(b)}}(x,\cE)=\chi^{(3)}(x) \frac{1}{T}\int_0^T \cE(x,t)\cdot \cE(x,t) dt \ \cE(x,t),  \label{Kerr_b}
\end{equation}
see, e.g., \cite{Stuart_93,Sutherland03}. For $\cE$ as in \eqref{monochromatic}, where $T=\pi/\omega$, we get the same for as in \eqref{Kerr_new}, i.e.
$$\cP_{\text{NL}}^{\text{(a)}}(x,e^{-\ri\omega t}E+\text{c.c.})=\cP_{\text{NL}}^{\text{(b)}}(x,e^{-\ri\omega t}E+\text{c.c.}).$$
 
To sum up, we may say that a solution $E:\R^3\to \C^3$ of \eqref{E:curl-curl} gives via \eqref{monochromatic} rise to a complete solution of the Maxwell system provided we consider the constitutive relation \eqref{Kerr_new} or \eqref{Kerr_b} instead of \eqref{Kerr}. 

With the notation
$$
V(x):=-\frac{\omega^2}{c^2}n^2(x), \qquad
\Gamma(x):=3\frac{\omega^2}{c^2}\chi^{(3)}(x)
$$
equation \eqref{E:curl-curl} reads
\beq \label{E:curl-curl-2}
\nabla \times \nabla \times E +V(x)E =  \frac{1}{3}\Gamma(x)\big(2|E|^2\,E+(E\cdot E)
\,\barl{E}\big) \mbox{ in } \R^3.
\eeq
Restricting to real valued solutions $E\in H^1(\R^3)$, equation \eqref{E:curl-curl-2} is equivalent to \eqref{E:curl-curl-gen} with $p=3$.

\medskip

\subsection{Structure of the paper}
The rest of the paper is structured as follows. In Section \ref{var_form} Theorem \ref{T:rad_sym} is first proved. Next, for the case of cylindrical symmetry of $V$ and $\Gamma$  a subspace of $X$ is chosen in which minimization of $J$ is possible. In Section~\ref{S:defoc} the strongly defocusing case (i.e. Theorem \ref{defoc}) is handled by the direct minimization method. Sections~\ref{sec:spectrum} and \ref{S:foc} treat the more delicate focusing case (i.e. Theorem \ref{foc}). In Section~\ref{sec:spectrum} we study the spectrum of the linear operator in \eqref{E:curl-curl-gen} and find examples of $V$ for which zero lies outside the spectrum. This is a necessary condition for our minimization approach. Finally, in Section~\ref{S:foc} $J$ is minimized on the so called Nehari-Pankov manifold within the symmetric subspace using the concentration-compactness principle.
 
\section{Variational formulation of \eqref{E:curl-curl-gen}} \label{var_form}

We begin with the definition of some spaces of vector valued functions $U:\R^3\to \R^3$.  For a measurable weight-function $\sigma:\R^3\to (0,\infty)$ the corresponding weighted $L^q$-space for $1\leq q<\infty$ is defined by 
$$
L_\sigma^q(\R^3) = \left\{ U: \R^3\to \R^3: \int_{\R^3}
\sigma(x)\,|U|^q\,dx <\infty\right\}
$$
with the norm 
$$
\|U\|_{\sigma,q} = \left(\int_{\R^3}
\sigma(x)\,|U|^q\,dx\right)^\frac{1}{q}.
$$
The space $H^1(\R^3)$ is defined by 
\begin{align*}
H^1(\R^3) &= \{ U:\R^3\to \R^3: U^i, \frac{\partial U^i}{\partial x_j} \in
L^2(\R^3) \mbox{ for } i,j=1,2,3\}, 
\end{align*} 
with the norm
$$
\|U\|_{H^1}^2 = \int_{\R^3} \sum_{i,j=1}^3 \left(\frac{\partial U^i}{\partial
x_j}\right)^2+ |U|^2\,dx.
$$
The space $H(\curl;\R^3)$ is defined by 
$$
H(\curl;\R^3) = \{ U:\R^3\to \R^3: U, \nabla\times U \in
L^2(\R^3)\}, 
$$
with the norm
$$
\|U\|_{H(\curl)}^2 = \int_{\R^3} |\nabla\times U|+ |U|^2\,dx
$$
and where $\nabla\times U$ is understood in the distributional sense, i.e., it satisfies $\int_{\R^3} \left(\nabla\times U\right)\cdot \phi \,dx = \int_{\R^3} U\cdot \curl\phi\,dx$ for all $C^\infty$-functions $\phi:\R^3\to\R^3$ with compact support. 

\medskip

Notice that for $U\in H^1(\R^3)$ we have the identity
$$
\int_{\R^3} |\nabla\times U|^2+(\nabla\cdot U)^2 \, dx = \int_{\R^3}
\sum_{i,j=1}^3 \left(\frac{\partial U^i}{\partial x_j}\right)^2\,dx
$$
with $\nabla\cdot U$ denoting the distributional divergence of $U$. Therefore
\begin{equation}
H(\curl;\R^3)\cap\{U: \nabla\cdot U=0\} = H^1(\R^3)\cap\{U:\nabla\cdot U=0\}.
\label{identity}
\end{equation}
This property will be used when we single out a suitable subspace of $H^1(\R^3)$, one in which we can solve
\eqref{E:curl-curl-gen}. For this purpose we first study the symmetries of
\eqref{E:curl-curl-gen}. 

\begin{lemma} \label{symmetry} Assume that the locally bounded measurable functions $V, \Gamma:\R^3\to \R^3$ are radially symmetric. 
\begin{itemize}
\item[(a)] If $U\in L_{loc}^p(\R^3;\R^3)$ is a distributional solution of \eqref{E:curl-curl-gen} and $M\in O(3)$
then $\tilde U(x):= M^TU(Mx)$ also solves \eqref{E:curl-curl-gen} in the sense of distributions. 
\item[(b)] Suppose $U:\R^3\to \R^3$ satisfies $U(x)=M^T U(Mx)$ for a.a. $x\in \R^3$ and all $M\in O(3)$. Then $U(x) = f(|x|)\frac{x}{|x|}$ for some $f:[0,\infty)\to \R$. If additionally $U\in L^1_{\rm{loc}}( \R^3)$ then $\nabla\times U=0$ in the sense of distributions.
\end{itemize}
\end{lemma}

\begin{proof} (a) Let $\phi: \R^3\to \R^3$ be a $C^\infty$-function with compact support, let $M\in O(3)$ and define 
$$
\psi(y) := M \phi(M^T y), \qquad y \in \R^3.
$$
Then a direct computation yields 
$$
\nabla\times\nabla\times \psi(y) = M (\nabla\times\nabla\times \phi)(M^T y)
$$
and thus 
$$
(\cL\psi)(y) =  M(\cL\phi)(M^Ty).
$$
Therefore, if $U\in L_{loc}^p(\R^3;\R^3)$ satisfies 
$$
\int_{\R^3} U(y) \cdot (\cL\psi)(y) -\Gamma(y) |U(y)|^{p-1} U(y)\cdot \psi(y)\,dy = 0 \mbox{ for all } \psi \in C_0^\infty(\R^3;\R^3)
$$
then 
\begin{eqnarray*}
\lefteqn{\int_{\R^3} \tilde U(x) \cdot (\cL\phi)(x) - \Gamma(x) |\tilde U(x)|^{p-1} \tilde U(x)\cdot \phi(x)\,dx } \\
& = & \int_{\R^3} M^T U(y) \cdot (\cL\phi)(M^Ty) - \Gamma(y) |U(y)|^{p-1} M^TU(y)\cdot \phi(M^Ty)\,dy \\
& = & \int_{\R^3} M^T U(y) \cdot M^T \cL\psi(y)- \Gamma(y) |U(y)|^{p-1} U(y)\cdot \psi(y)\,dy \\
& = & \int_{\R^3} U(y) \cdot \cL\psi(y)- \Gamma(y) |U(y)|^{p-1} U(y)\cdot \psi(y)\,dy \\
& = & 0.
\end{eqnarray*}
(b) Let $x\in\R^3$ be such that $U(x)=M^TU(Mx)$ for all $M\in O(3)$. Then $U(x)=MU(x)$ for all those rotations $M$ which leave $x$ fixed, i.e., for all rotations around the axis $\R x$. Hence $U(x)\in\R x$ and we may write $U(x) = f(|x|)\frac{x}{|x|}$. Under the assumption $U\in L^1_{loc}(\R^3)$ we see that $f\in L^1(I)$ for any compact interval $I\subset (0,\infty)$. Therefore we may define the function $F(r):= \int_1^r f(s)\,ds$ for $r>0$  which is absolutely continuous in $\R^+$. Moreover, for any $R>1$ using polar coordinates and Fubini's theorem we see that
\begin{align*}
\int_{B_R(0)} |F(|x|)| \,dx &= 4\pi \int_0^R \left| \int_1^r f(t)\,dt \right| r^2\,dr \\
& \leq 4\pi \int_0^1 \int_r^1 |f(t)|\,dt\, r^2 \,dr + 4\pi \int_1^R \int_1^r |f(t)|\,dt\, r^2\,dr \\
& \leq \frac{4\pi}{3} \int_0^1 |f(t)|t^3\,dt + \frac{4\pi R^3}{3} \int_1^R |f(t)|\,dt \\
& \leq \frac{4\pi R^3}{3} \int_0^R |f(t)|t^2\,dt \\
& = \frac{R^3}{3}\int_{B_R(0)} |U(x)|\,dx<\infty
\end{align*} 
since $U\in L^1_{loc}(\R^3)$. Hence the function $F(|\cdot|)$ belongs to $L^1_{loc}(\R^3)$ and due to the absolute continuity of $F$ it has the strong derivative $U(x)$ almost everywhere. Since both $F(|\cdot|)$  and $U$ are $L^1_{loc}(\R^3)$, one can see that $U(x)=\nabla\left(F(|x|)\right)$ in $\R^3$ in the weak sense. This implies $\nabla\times U=0$ in the distributional sense.
\end{proof}

\noindent
{\em Proof of Theorem~\ref{T:rad_sym} :} Suppose $U\in L^p_{loc}(\R^3)$ is a distributional solution of \eqref{E:curl-curl-gen}. Lemma~\ref{symmetry} shows that the requirement of full radial symmetry of $V$ and $\Gamma$ and the solution symmetry $U(x)=M^T U(Mx)$ for a.a. $x\in \R^3$ and all $M\in O(3)$ reduces \eqref{E:curl-curl-gen} to the algebraic equation
$$
V(x) U = \Gamma(x)|U|^{p-1} U \mbox{ in } \R^3.
$$
Provided $0 \leq V \Gamma^{-1}$ the function $U$ has the form \eqref{E:U_rad_sym}. Now let us reversely assume that $U$ has the form \eqref{E:U_rad_sym}. Since $0 \leq V \Gamma^{-1} \in L^\frac{p}{p-1}_{loc}(\R^3)$ we see that $U\in L^p_{loc}(\R^3)$ and in particular $U\in L^1_{loc}(\R^3)$. Moreover, with an absolutely continuous function $F:(0,\infty)\to \R$ given by 
$$
F(t)= \int_1^t s(\tau) \left( \frac{\tilde V(\tau)}{\tilde\Gamma(\tau)}\right)^{1/(p-1)} \,d\tau, \qquad t>0
$$
we have $U(x)=\nabla\left(F(|x|)\right)$ in the distributional sense. As in Lemma~\ref{symmetry} we find $F(|\cdot|)\in L^1_{loc}(\R^3)$ and, moreover, $\nabla\times U=0$. Hence $U$ solves \eqref{E:curl-curl-gen}. Finally, the assumption $(V\Gamma^{-1})^\frac{2}{p-1}$, $(V\Gamma^{-1})^\frac{p+1}{p-1}\Gamma\in L^1(\R^3)$ implies that any $U$ defined by \eqref{E:U_rad_sym} belongs to the space $X$ and thus is a critical point of $J$.
\qed

\medskip

Although the $H(\curl;\R^3)$ solutions in Theorem~\ref{T:rad_sym} are valid localized solutions of \eqref{E:curl-curl-gen}, in the rest of the paper we consider solutions that are not gradient fields.

\medskip

Since the requirement of full radial symmetry does not lead to interesting solutions of \eqref{E:curl-curl-gen}, we look for solutions which are invariant only under a subgroup of $O(3)$ (this idea is due to Azzollini et. al. \cite{ABDF06}). For this we define the following copy of $SO(2)$ as a subset of $O(3)$
$$
G_0 := \left\{\begin{pmatrix} 
\cos \alpha  & -\sin\alpha & 0 \\
\sin\alpha & \cos\alpha & 0 \\
0 & 0 & 1
\end{pmatrix} : \alpha \in \R\right\}.
$$
Assume that the measurable weight $\sigma:\R^3\to (0,\infty)$ satisfies $\sigma(Mx)=\sigma(x)$ for all $x\in \R^3$ and all $M\in G_0$. Then the group $G_0$ operates isometrically on $L^q_\sigma(\R^3)$, on $H(\curl;\R^3)$ and on $H^1(\R^3)$ by the group action $U
\mapsto M^T U(M\cdot)$. Due to this result we can now define the corresponding $G_0$-fixed point subspaces of $L^q_\sigma(\R^3)$, $H(\curl;\R^3)$ and $H^k(\R^3)$, $k\in \N$ by
\begin{align*}
L_{\sigma,G_0}^q(\R^3) & = \{U\in L^q_\sigma(\R^3): U(x)=M^TU(Mx) \quad \forall
x\in\R^3, \forall  M\in G_0\}, \\
H_{G_0}(\curl;\R^3) & = \{U\in H(\curl;\R^3): U(x)=M^TU(Mx) \quad \forall
x\in\R^3, \forall  M\in G_0\}, \\
H_{G_0}^k(\R^3) & = \{U\in H^k(\R^3): U(x)=M^TU(Mx) \quad \forall
x\in\R^3, \forall  M\in G_0\}, \quad k\in \N, \\
X_{G_0} &=H_{G_0}(\curl;\R^3)\cap L_{|\Gamma|}^{p+1}(\R^3).
\end{align*}
Observe that the functional $J$ is invariant under the action of $G_0$. Thus, by Palais' principle of symmetric criticality \cite{palais}, \cite{willem}, every critical point of $J|_{X_{G_0}}$ is also a critical point of $J$ on $X$. Next we want to restrict the spaces $L^q_{\sigma,G_0}(\R^3)$, $H_{G_0}(\curl;\R^3)$ and $H^k_{G_0}(\R^3)$ even further. In order to do so we need two lemmas -- the first one being analogous to Lemma~\ref{symmetry}. We omit the proofs because they are contained in Lemma~1 and Proposition~1 in \cite{ABDF06}.

\begin{lemma} \label{symmetry_zyl} Suppose a measurable function $U:\R^3\to\R^3$ satisfies 
$U(x)=M^T U(Mx)$ for a.a. $x\in \R^3$ and all $M\in G_0$. Then there are unique measurable functions $Q, S, T: \R^3\to \R^3$ such that 
$$
U(x) = Q(x)+ S(x)+ T(x)
$$ 
with 
\beq
Q(x) = \frac{q(r,x_3)}{r}\begin{pmatrix} -x_2 \\ x_1 \\ 0 \end{pmatrix}, \quad S(x)=
\frac{s(r,x_3)}{r}\begin{pmatrix} x_1 \\ x_2 \\ 0 \end{pmatrix}, \quad T(x) = 
\begin{pmatrix} 0 \\ 0 \\ t(r,x_3) \end{pmatrix}.
\label{decomp}
\eeq
where $q, s, t : (0,\infty)\times \R\to \R$  are measurable function. 
If $U\in L^q_\sigma(\R^3)$ or $H^1(\R^3)$ then $Q,S,T\in L^q_\sigma(\R^3)$ or $H^1(\R^3)$, respectively.
\end{lemma}

\begin{lemma} Let $Y$ be either the spaces $L_{\sigma,G_0}^q(\R^3)$, $H_{G_0}(\curl;\R^3)$ or $H_{G_0}^k(\R^3)$, $k\in\N$ and define the map 
$$
g_1 : \left\{ \begin{array}{rcl} 
Y & \to & Y, \vspace{\jot} \\
U= Q+S+T & \mapsto & Q-S-T.
\end{array} \right.
$$
The map $g_1$ is a linear isometry and satisfies $g_1\circ g_1=Id$. Hence $G_1= \{\Id, g_1\}$ is a group of order 2. Moreover, the functional $J|_{X_{G_0}}$ is invariant under the action of $G_1$. 
\label{def_g1}
\end{lemma}
\noindent {\bf Remark.} The proof of the isometry and invariance statement relies on the fact that pointwise $|U|^2=|Q|^2+|S|^2+|T|^2$ and $|\nabla U|^2=|\nabla Q|^2+|\nabla S|^2+|\nabla T|^2$, cf. \cite{ABDF06}. However, for the curl only $|\nabla\times U|^2=|\nabla\times Q|^2+|\nabla\times(S +T)|^2$ holds. But this is sufficient for our claim. 


\medskip

This result allows to define the spaces 
\begin{align*}
L_{\sigma,G_1}^q(\R^3) &= \{U\in L_{\sigma,G_0}^q(\R^3): g_1 U=U \}, \\
H_{G_1}(\curl;\R^3) & = \{U\in H_{G_0}(\curl;\R^3): g_1 U = U \}, \\
H_{G_1}^k(\R^3) &= \{U\in H_{G_0}^k(\R^3): g_1 U=U \}, \quad k\in\N, \\
X_{G_1} &= H_{G_1}(\curl;\R^3)\cap L_{|\Gamma|}^{p+1}(\R^3).
\end{align*}
All the spaces with the suffix $G_1$ may be thought of as the subspaces of $L^q_\sigma(\R^3)$, $H(\curl;\R^3)$ and $H^k(\R^3)$ consisting of vector fields of the form \eqref{symmetric_form}. Again, Palais' principle of symmetric criticality ensures that every critical point of $J|_{X_{G_1}}$ is also a critical point of $J$ on $X$. Finally, note that 
\begin{equation}
H_{G_1}(\curl;\R^3)= H_{G_1}^1(\R^3)
\label{identity_G1}
\end{equation}
because the members of both spaces have vanishing divergence, cf. \eqref{identity}.

\medskip

%
%

To summarize the results of this section recall that the energy functional
related to \eqref{E:curl-curl-gen} is
$$
J[U]=\int_{\R^3} \frac{1}{2}(|\nabla \times U|^2+V(x) |U|^2) -
\frac{\Gamma(x)}{p+1}|U|^{p+1}\,dx,
$$
which is well defined on $X=H(\curl;\R^3)\cap L_{|\Gamma|}^{p+1}(\R^3)$. Due to Lemma~\ref{symmetry_zyl}, Palais' principle of symmetric criticality
(cf. Palais~\cite{palais}, Willem~\cite{willem}) and \eqref{identity_G1} can seek critical points of
the functional $J$ restricted to the subspace $X_{G_1}=H_{G_1}^1(\R^3)\cap L_{|\Gamma|}^{p+1}(\R^3)$ and these critical
points will be solutions of \eqref{E:curl-curl-gen}. The elements of the subspace $H^1_{G_1}(\R^3)$  have the favorable property of vanishing divergence.




\section{Ground states in the defocusing case}\label{S:defoc}

We assume $p>1$ and
$$
\Gamma(x) \leq -C(1+|x|)^\alpha \mbox{ in } \R^3 \mbox{ with } \alpha >
\frac{3}{2}(p-1) \mbox{ and } C>0. \leqno {\mbox{(H-defoc)}}
$$
We work in the following reflexive Banach space
$$
X_{G_1} := H_{G_1}^1(\R^3)\cap L_{|\Gamma|}^{p+1}(\R^3)
$$
where the norm on $X_{G_1}$ is given by 
$$
\|U\|_X := \|U\|_{H^1}+\|U\|_{|\Gamma|,p+1}.
$$
The basic tool for proving existence of ground states in the defocusing case is
the following embedding result, which is due to Benci, Fortunato~\cite{be_fo2}.

\begin{lemma} \label{embed} 
Assume $p>1$ and {\rm (H-defoc)}. Then the space $L_{|\Gamma|}^{p+1}(\R^3)$ embeds continuously into
$L^2(\R^3)$ and the space $X_{G_1}$ embeds compactly into $L^2(\R^3)$. 
\end{lemma} 

\begin{proof} Let $\beta= \frac{p+1}{p-1}$ and $\beta'=\frac{p+1}{2}$. Then 
\begin{align*}
\int_{\R^3} |U|^2\,dx &= \int_{\R^3} |\Gamma(x)|^{-\frac{1}{\beta'}} |\Gamma(x)|^\frac{1}{\beta'} |U|^2\,dx \\
& \leq \left(\int_{\R^3} |\Gamma(x)|^{-\frac{2}{p-1}}\,dx\right)^\frac{p-1}{p+1} \left(\int_{\R^3} |\Gamma(x)| |U|^{p+1}\,dx\right)^\frac{2}{p+1},
\end{align*}
and the first integral is finite since by assumption (H-defoc) $|\Gamma(x)|^{-\frac{2}{p-1}} \leq C (1+|x|)^{-\frac{2\alpha}{p-1}}$ and $-\frac{2\alpha}{p-1}<-3$. This proves the first part of the claim. For the second part, let us define the positive and continuous radially symmetric function $\rho:\R^3\to (0,\infty)$ by setting 
$\rho(x)=1$ for $|x|\leq 1$ and $\rho(x)=|x|^\gamma$ with $\gamma>0$ so small that $\gamma\beta - \frac{2\alpha}{p-1}<3$. Then $\rho(x)\to \infty$ as $|x|\to \infty$ and $\int_{\R^3} \rho(x)^\beta (1+|x|)^{-\frac{2\alpha}{p-1}}\,dx <\infty$. We obtain 
\begin{align*} 
\|U\|_{\rho,2}^2 &= \int_{\R^3} \rho(x)|\Gamma(x)|^{-\frac{1}{\beta'}} |\Gamma(x)|^\frac{1}{\beta'}|U|^2\,dx \\
& \leq \left(\int_{\R^3} \rho(x)^\beta |\Gamma(x)|^{-\frac{2}{p-1}}\,dx\right)^\frac{p-1}{p+1} \left(\int_{\R^3} |\Gamma(x)| |U|^{p+1}\,dx\right)^\frac{2}{p+1} \\
& \leq C \|U\|^2_{|\Gamma|,p+1}.
\end{align*}
This shows that $L_{|\Gamma|}^{p+1}(\R^3)$ embeds continuously into $L^2_{\rho}(\R^3)$. Finally, by Theorem 3.1 of Benci, Fortunato \cite{be_fo2} we have that 
$H^1(\R^3)\cap L^2_{\rho}(\R^3)$ embeds compactly into $L^2(\R^3)$. Both facts together imply the second claim of the lemma.
\end{proof}

\begin{lemma} Assume $p>1$, {\rm (H-defoc)} and $V\in L^\infty(\R^3)$. Then the functional $J$ is a weakly lower-semicontinuous, coercive $C^1$-functional on $X_{G_1}$ and hence has a minimizer.
\end{lemma}

\begin{proof} Since 
$$
J_1[U] = \int_{\R^3} \frac{1}{2}|\nabla \times U|^2-
\frac{\Gamma(x)}{p+1}|U|^{p+1}\,dx
$$
is convex on $X_{G_1}$ and 
$$
J_2[U]= \int_{\R^3} \frac{V(x)}{2} |U|^2\,dx
$$ 
is weakly continuous on $X_{G_1}$ by Lemma~\ref{embed} we find that the functional
$J=J_1+J_2$ is weakly lower-semicontinuous on $X_{G_1}$. Moreover, there exist
constants $K_1,\ldots,K_5>0$ such that the following estimates hold for $U\in
X_{G_1}$: 
\begin{align*}
J[U] &\geq  \frac{1}{2}\|\nabla\times U\|_2^2 + \frac{1}{p+1}
\|U\|_{|\Gamma|,p+1}^{p+1}- \frac{\|V\|_\infty}{2} \|U\|_2^2 \\
& \geq  \frac{1}{2}\|\nabla\times U\|_2^2 + \frac{1}{p+1}
\|U\|_{|\Gamma|,p+1}^{p+1}- K_1\|U\|_{|\Gamma|,p+1}^2 \\
& \geq \frac{1}{2}\|\nabla\times U\|_2^2+K_2 \|U\|_{|\Gamma|,p+1}^2-K_3 \\
& \geq \frac{1}{2}\|\nabla\times U\|_2^2 + K_4 \|U\|_2^2 + \frac{K_2}{2} \|U\|_{|\Gamma|,p+1}^2-K_3
\\
& \geq K_5 \|U\|_X^2 -K_3,
\end{align*}
which shows the coercivity of $J$. It is clear that the quadratic parts of the functional $J$ are $C^1$ and it is standard (cf. Struwe~\cite{struwe}) to verify that the functional 
$\int_{\R^3} \Gamma(x)|U|^{p+1}$ has a G\^{a}teaux derivative which depends continuously on $U\in X_{G_1}$. Hence $J$ is a $C^1$-functional on $X_{G_1}$ and the minimizer of $J$ is a weak solution of \eqref{E:curl-curl-gen}. 
\end{proof}

\noindent
{\em Proof of Theorem~\ref{defoc}:} We set $U_0(x)= sW(tx)$ for some vector-valued function $W\in C_0^\infty(\R^3)$ and take $s,t>0$. Since $\esssup_{\R^3} V <0$ we obtain
\begin{align*}
J[U_0] & = \int_{\R^3} \frac{1}{2}|\nabla\times U_0(x)|^2-\frac{\Gamma(x)}{p+1}|U_0(x)|^{p+1} +\frac{V(x)}{2}|U_0(x)|^2\,dx \\
&\leq t^{-3}s^2 \int_{\R^3} \frac{t^2 }{2}|\nabla\times W(y)|^2-\frac{s^{p-1}\Gamma(y/t)}{p+1}|W(y)|^{p+1}+ \frac{\esssup_{\R^3} V}{2} |W(y)|^2\,dy \\
& <0 
\end{align*}
provided we first choose $t>0$ so small that $\int_{\R^3} t^2|\nabla\times W(y)|^2+ (\esssup_\R^3 V)|W(y)|^2\,dy<0$ and then choose $s>0$ sufficiently small. Thus the minimizer of $J$ over $X_{G_1}$ is non-trivial and therefore a ground state of
\eqref{E:curl-curl-gen} within $X_{G_1}$.
\qed

\section{Spectrum of the linear operator $\mathcal{L}$} \label{sec:spectrum}

In the focusing case we can only show the existence of ground states (cf. Section~\ref{S:foc}) when zero
does not lie in the spectrum of the linear operator 
$$
\mathcal{L} := (\nabla\times\nabla\times) + V(r,x_3).
$$
Of course an easy example is given by the class of potentials $V=V(r,x_3)$ with $\essinf V_{\R^3}>0$. However, since $V(x)$ is proportional to $-n^2(x)$ with $n(x)$ being the refractive index, the physically interesting case consists of functions $V$ which are negative (or at least have non-vanishing negative part). In this section we construct potentials $V$ with non-vanishing negative part and where $0$ lies in a spectral gap of the operator $\mathcal{L}$, cf. Lemma~\ref{l:ex1}, Lemma~\ref{l:ex2}, Lemma~\ref{l:ex3}. 

\medskip

The construction of such examples needs various preprations. We consider ${\mathcal L}$ as an operator defined on 
$$
D(\mathcal{L})= H^2_{G_1}(\R^3)\subset L^2_{G_1}(\R^3)
$$ 
and we will show in Lemma~\ref{lem:self} that $\cL$ is a selfadjoint operator, whose spectrum has a particular additive structure whenever the potential is separable, i.e., $V(r,x_3)= W(r)+P(x_3)$, cf. Lemma~\ref{lem:separabel}. The key to these results is the following observation: if 
$$
U(x) = u(r,x_3) \begin{pmatrix} -x_2 \\ x_1 \\ 0 \end{pmatrix} \mbox{ with } r=\sqrt{x_1^2+x_2^2}
$$
then 
\begin{equation}
\label{beziehung}
\cL U(x) =\Bigl((Lu)(r,x_3)\Bigr) \begin{pmatrix} -x_2 \\ x_1 \\ 0 \end{pmatrix} 
\end{equation}
with 
\begin{equation}
L = -\frac{1}{r^3}\frac{\partial}{\partial r}\left(r^3\frac{\partial}{\partial r}\right) - \frac{\partial^2}{\partial x_3^2} + V(r,x_3)
\label{skalarer_operator}
\end{equation}
where the first two terms correspond to a five-dimensional Laplacian with cylindrical symmetry. Let us now start with the detailed analysis of the operators and their spectra.

\medskip

Define the maps
$$
\Psi_{\rm{rad}}: \left\{ \begin{array}{rcl} 
\R^4 & \to & \R, \vspace{\jot}\\
(y_1,\ldots,y_4) & \mapsto & \sqrt{y_1^2+\ldots+y_4^2}\end{array}\right.
\quad 
\Psi: \left\{ \begin{array}{rcl} 
\R^5 & \to & \R^2, \vspace{\jot}\\
(y_1,\ldots,y_5) & \mapsto & (\Psi_{\rm{rad}}(y_1,\ldots,y_4), y_5)\end{array}\right.
$$
In the following we use the index \emph{rad} for spaces of functions $u:(0,\infty)\to \R$ of the single radial variable $r$; the index \emph{cyl} refers to spaces of functions $u: (0,\infty)\times \R\to \R$ of two variables $r, x_3$). For the following Hilbert spaces we also denote in brackets the measure with respect to which integration is performed.
\begin{align*}
L^2_{\rm{rad}}(r^3dr) &= \left\{u:(0,\infty) \to \R: u\circ\Psi_{\rm{rad}} \in L^2(\R^4)\right\}\\
&= \left\{ u:(0,\infty) \to \R: u \in L^2_{r^3}(0,\infty)\right\}, \\
L^2_{\rm{cyl}}(r^3drdx_3) &= \left\{u:(0,\infty)\times \R \to \R: u\circ\Psi \in L^2(\R^5)\right\}\\
&=\left\{ u:(0,\infty) \to \R: u \in L^2_{r^3}((0,\infty)\times \R)\right\}, \\
H^1_{\rm{rad}}(r^3dr) &=\left\{u:(0,\infty) \to\R: u\circ\Psi_{\rm{rad}} \in H^1(\R^4)\right\}\\
&=\left\{ u: (0,\infty)\to \R: u, u' \in  L^2_{r^3}(0,\infty)\right\}, \\
H^1_{\rm{cyl}}(r^3drdx_3) &=\left\{u:(0,\infty)\times\R\to\R: u\circ\Psi \in H^1(\R^5)\right\} \\
&= \left\{u:(0,\infty)\times\R\to\R: u, \frac{\partial u}{\partial r}, \frac{\partial u}{\partial x_3} \in L^2_{r^3}((0,\infty)\times \R)\right\}, \\
H^2_{\rm{rad}}(r^3dr) &=\left\{u:(0,\infty) \to\R: u\circ\Psi_{\rm{rad}} \in H^2(\R^4)\right\}\\
&=\left\{ u: (0,\infty)\to \R: u, u', \frac{u'}{r}, u'' \in  L^2_{r^3}(0,\infty)\right\}, \\
H^2_{\rm{cyl}}(r^3drdx_3) &=\left\{u:(0,\infty)\times\R\to\R: u\circ\Psi \in H^2(\R^5)\right\} \\
&= \left\{u:(0,\infty)\times\R\to\R: u, \frac{\partial u}{\partial r}, \frac{\partial u}{\partial x_3}, \frac{1}{r}\frac{\partial u}{\partial r}, \frac{\partial^2 u}{\partial r^2}, \frac{\partial^2 u}{\partial x_3^2} \in L^2_{r^3}((0,\infty)\times \R)\right\}.
\end{align*}
These identities may be well known. For the sake of clarity we explain the last one for $H^2_{\rm{cyl}}(r^3drdx_3)$ on the level of the derivatives of highest order: $u\circ\Psi$ has second order derivatives in $L^2(\R^5)$ if and only if for all $i,j\in \{1,2,3,4\}$ we have 
\begin{equation}
\left(\frac{\partial^2 u}{\partial r^2}-\frac{1}{r}\frac{\partial u}{\partial r}\right)\frac{y_i y_j}{r^2}+\frac{\partial u}{\partial r}\frac{\delta_{ij}}{r}, \frac{\partial^2 u}{\partial x_3^2}, \frac{\partial^2 u}{\partial r\partial x_3} \frac{y_i}{r} \in L^2(\R^5).
\label{second_derivatives}
\end{equation}
In view of the fact that $\int_{\R^3} \sum_{i,j = 1}^3 \left(\frac{\partial^2 v(x)}{\partial x_i \partial x_j}\right)^2\,dx = \int_{\R^3} \left(\Delta v(x)\right)^2\,dx$ for $v\in C_0^\infty(\R^3)$ we see that
$$
u\circ\Psi, \Delta (u\circ \Psi)\in L^2(\R^5) \Longleftrightarrow u\circ \Psi\in H^2(\R^5)
$$
Hence, for \eqref{second_derivatives} it is sufficient to have 
\begin{equation}
u,\; \frac{1}{r}\frac{\partial u}{\partial r},\; \frac{\partial^2 u}{\partial r^2},\; \frac{\partial^2 u}{\partial x_3^2} \in L^2_{r^3}((0,\infty)\times \R)
\label{sufficient_second_order}
\end{equation}
since the $L^2$-norm of $\frac{\partial^2 u}{\partial r \partial x_3}$ may be estimated by the $L^2$-norm of $\Delta (u\circ \Psi)$, i.e., by sums of $L^2_{r^3}((0,\infty)\times \R)$-norms of $\frac{1}{r}\frac{\partial u}{\partial r}$, $\frac{\partial^2 u}{\partial r^2}$, $\frac{\partial^2 u}{\partial x_3^2}$. 
Let us also show that \eqref{sufficient_second_order} is necessary. If we square all the entries in \eqref{second_derivatives} and add them up then we see that 
$$
\int_0^\infty \int_{-\infty}^\infty \left(\left(\frac{\partial^2 u}{\partial r^2}\right)^2+\frac{3}{r^2} \left(\frac{\partial u}{\partial r}\right)^2+ \left(\frac{\partial^2 u}{\partial x_3^2}\right)^2+ \left(\frac{\partial^2 u}{\partial r \partial x_3}\right)^2\right) r^3\,drdx_3 < \infty,
$$
which implies that \eqref{sufficient_second_order} is also necessary. 

\medskip

The spaces carry natural inner products. Instead of listing all of them we just write out the ones for $H^2_{\rm{rad}}(r^3dr)$ and $H^2_{\rm{cyl}}(r^3drdx_3)$: 
\begin{align*}
\langle u,v \rangle_{H^2_{\rm{rad}}} &= \int_0^\infty \left(uv+ u'v'+\frac{1}{r^2} u'v'+u''v''\right) r^3dr \\
\langle u,v \rangle_{H^2_{\rm{cyl}}} &= \int_0^\infty \int_{-\infty}^\infty \left(uv+ \frac{\partial u}{\partial r}\frac{\partial v}{\partial r}+ \frac{\partial u}{\partial x_3}\frac{\partial v}{\partial x_3}
+\frac{1}{r^2}\frac{\partial u}{\partial r}\frac{\partial v}{\partial r}
+\frac{\partial^2 u}{\partial r^2}\frac{\partial^2 v}{\partial r^2}+ \frac{\partial^2 u}{\partial x_3^2}\frac{\partial^2 v}{\partial x_3^2}\right)r^3drdx_3.
\end{align*}

The above identities of spaces have the following implication.

\begin{lemma}[Hardy's inequality]\
\begin{itemize}
\item[(i)] There exists a constant $C>0$ such that for all $u\in H^1_{\rm{cyl}}(r^3drdx_3)$
$$
\int_0^\infty \int_{-\infty}^\infty \frac{u^2}{r^2} \,r^3drdx_3 \leq \int_0^\infty \int_{-\infty}^\infty \left(\left(\frac{\partial u}{\partial r}\right)^2+ \left(\frac{\partial u}{\partial x_3}\right)^2\right) \, r^3drdx_3.
$$
\item[(ii)] There exists a constant $C>0$ such that for all $u\in H^2_{\rm{cyl}}(r^3drdx_3)$
$$
\int_0^\infty \int_{-\infty}^\infty \frac{1}{r^2}\left(\frac{\partial u}{\partial x_3}\right)^2 \,r^3drdx_3 \leq \int_0^\infty \int_{-\infty}^\infty \left(\left(\frac{\partial^2 u}{\partial r^2}\right)^2+ \frac{1}{r^2}\left(\frac{\partial u}{\partial r}\right)^2+ \left(\frac{\partial^2 u}{\partial x_3^2}\right)^2\right) \, r^3drdx_3.
$$
\end{itemize}
\label{hardy}
\end{lemma}

\begin{proof} We use the identities of spaces as explained above. Part (i) can be found as Theorem~C in \cite{barbatis_et_al} set up in $\R^5$ where $r=\dist(y,K)$ and $K=\{y\in \R^5: y_1=y_2=y_3=y_4=0\}$. Note that $H_0^1(\R^5\setminus K)= H_0^1(\R^5)$, cf. Theorem~2.43 in \cite{martio_et_al}, because as a subset of $\R^5$ the set $K$ has zero $2$-capacity, cf. Section~4.7.2 in \cite{evans_gariepy}. Part (ii) is a consequence of (i) when applied to $\partial U/\partial y_5 \in H_0^1(\R^5)$.
\end{proof}

\begin{lemma} \label{spaces} The following identity holds between the group invariant spaces (denoted with suffix $G_1$) and the spaces of scalar functions with cylindrical symmetry (denoted with suffix {\rm cyl}):
\begin{align*}
L^2_{G_1}(\R^3) &= \left\{ u(r,x_3) \begin{pmatrix} -x_2 \\ x_1 \\ 0 \end{pmatrix}: u \in L^2_{\rm{cyl}}(r^3drdx_3)\right\} \\
H^k_{G_1}(\R^3) &= \left\{ u(r,x_3) \begin{pmatrix} -x_2 \\ x_1 \\ 0 \end{pmatrix}: u \in H^k_{\rm{cyl}}(r^3drdx_3)\right\}, \quad k=1,2.
\end{align*}
\end{lemma}
\begin{proof} Let 
$$
U(x_1,x_2,x_3) = u(r,x_3) \begin{pmatrix} -x_2 \\ x_1 \\ 0 \end{pmatrix}
$$
The identity of the $L^2$-spaces is obvious since 
$$
\int_{\R^3}|U(x)|^2\,dx= 2\pi\int_0^\infty \int_{-\infty}^\infty|u(r,x_3)|^2 r^3 drdx_3.
$$
Next we prove the identity of the $H^1$-spaces. Clearly $U\in H^1(\R^3)$ if and only if $x_iu \in H^1(\R^3)$ for $i=1,2$, i.e., if and only if 
$$
u\delta_{ij}+ \frac{x_ix_j}{r} \frac{\partial u}{\partial r},\; x_i \frac{\partial u}{\partial x_3}, \; x_iu \in L^2(\R^3).
$$
By squaring, summing from $i,j=1,2$ and rearranging terms this in turn is equivalent to 
$$
\int_0^\infty \int_{-\infty}^\infty \left( u^2+\left(u+r\frac{\partial u}{\partial r}\right)^2 + r^2\left(\frac{\partial u}{\partial x_3}\right)^2+ r^2 u^2\right) r\,drdx_3 < \infty.
$$
The above is equivalent to $ \frac{u}{r}, \frac{\partial u}{\partial r}, \frac{\partial u}{\partial x_3}, u \in L^2_{r^3}((0,\infty)\times \R)$. Hardy's inequality of Lemma~\ref{hardy} tells us that the $L^2_{r^3}((0,\infty)\times \R)$-norm of the first term is bounded by the norm of the remaining terms, and hence $U\in H^1(\R^3)$ if and only if $u\in H^1_{\rm{cyl}}(r^3drdx_3)$. Finally, let us prove the identity for the $H^2$-spaces. The second derivatives of $U$ lie in $H^2(\R^3)$ if and only if 
$$
\delta_{ij} \frac{x_k}{r} \frac{\partial u}{\partial r} + \delta_{ik} \frac{x_j}{r} \frac{\partial u}{\partial r}+\delta_{jk} \frac{x_i}{r}\frac{\partial u}{\partial r}- \frac{x_ix_jx_k}{r^3}\frac{\partial u}{\partial r} + \frac{x_ix_jx_k}{r^2}\frac{\partial^2 u}{\partial r^2}, \; \delta_{ij} \frac{\partial u}{\partial x_3} + \frac{x_ix_j}{r} \frac{\partial^2 u}{\partial r\partial x_3} \in L^2(\R^3)
$$
for $i,j,k=1,2$. By squaring, summing from $i,j,k=1,2$ and rearranging this becomes 
\begin{align*}
\int_0^\infty \int_{-\infty}^\infty \left(3\left(\frac{\partial u}{\partial r}\right)^2 + \left(2\frac{\partial u}{\partial r}+r\frac{\partial^2 u}{\partial r^2}\right)^2 \right) r\,drdx_3 &< \infty \\
\int_0^\infty \int_{-\infty}^\infty\left(\left(\frac{\partial u}{\partial x_3}\right)^2+ \left(\frac{\partial u}{\partial x_3}+r\frac{\partial^2 u}{\partial r\partial x_3}\right)^2\right)r\,drdx_3 &< \infty \\
\int_0^\infty \int_{-\infty}^\infty \left(\frac{\partial^2 u}{\partial x_3^2}\right)^2 r^3 \,drdx_3 &< \infty
\end{align*}
Therefore a necessary and sufficient condition for $U \in H^2(\R^3)$ is given by
$$
\int_0^\infty \int_{-\infty}^\infty \left(\frac{1}{r^2}\left(\frac{\partial u}{\partial r}\right)^2+ \left(\frac{\partial^2 u}{\partial r^2}\right)^2+\frac{1}{r^2}\left(\frac{\partial u}{\partial x_3}\right)^2 +\left(\frac{\partial^2 u}{\partial r\partial x_3}\right)^2+\left(\frac{\partial^2 u}{\partial x_3^2}\right)^2 \right)r^3\,drdx_3 <\infty
$$
By the relation $\|D^2 u\|_{L^2(\R^3)}= \|\Delta u\|_{L^2(\R^3)}$ and by Hardy's inequality of Lemma~\ref{hardy} the above is equivalent to 
$$
\int_0^\infty \int_{-\infty}^\infty \left(\frac{1}{r^2}\left(\frac{\partial u}{\partial r}\right)^2+ \left(\frac{\partial^2 u}{\partial r^2}\right)^2+\left(\frac{\partial^2 u}{\partial x_3^2}\right)^2 \right)r^3\,drdx_3 <\infty
$$
which means that $U$ has second derivatives in $L^2(\R^3)$ if and only if $\frac{1}{r}\frac{\partial u}{\partial r}, \frac{\partial^2 u}{\partial r^2}, \frac{\partial^2 u}{\partial x_3^2} \in L^2_{r^3}((0,\infty)\times \R)$. In view of the definition of $H^2_{\rm{cyl}}(r^3drdx_3)$ this establishes the claim.
\end{proof}

\begin{lemma} \label{l:self_l} Let $V\in L^\infty(\R^3)$ and suppose $V=V(r,x_3)$ has cylindrical symmetry. Then the operator $L: D(L):=H^2_{\rm{cyl}}(r^3drdx_3) \subset L^2_{\rm{cyl}}(r^3drdx_3)\to L^2_{\rm{cyl}}(r^3drdx_3)$ given by \eqref{skalarer_operator} is selfadjoint.
\end{lemma}
\begin{proof} For $u\in D(L)$ we have 
$$
(Lu)\circ \Psi = -\Delta (u\circ\Psi) + (Vu)\circ\Psi,
$$
i.e., $L$ coincides with the five-dimensional Schr\"odinger operator $-\Delta+V$ in the space of functions with cylindrical symmetry.
\end{proof}

\begin{lemma} \label{lem:self}
Let $V\in L^\infty(\R^3)$ and suppose $V=V(r,x_3)$ has cylindrical symmetry. The operator $\mathcal{L} := (\nabla\times\nabla\times) + V(r,x_3)$ defined on $D(\mathcal{L})= H^2_{G_1}(\R^3)\subset L^2_{G_1}(\R^3)\to L^2_{G_1}(\R^3)$ is selfadjoint and $\sigma(\cL)=\sigma(L)$.
\end{lemma}
\begin{proof} First we check the symmetry of $\cL$. Let $U,\tilde U\in D(\cL)$, i.e., by Lemma~\ref{spaces}, 
$$
U(x)=u(r,x_3) \begin{pmatrix} -x_2 \\ x_1 \\ 0 \end{pmatrix}, \quad \tilde U(x)=\tilde u(r,x_3) \begin{pmatrix} -x_2 \\ x_1 \\ 0 \end{pmatrix}
$$
for some $u,\tilde u\in D(L)$. Thus
\begin{align}
\langle \cL U,\tilde U\rangle_{L^2(\R^3)} &= \left\langle  \begin{pmatrix} -x_2 \\ x_1 \\ 0 \end{pmatrix}(Lu)(r,x_3), \begin{pmatrix} -x_2 \\ x_1 \\ 0 \end{pmatrix} \tilde u(r,x_3)\right\rangle_{L^2(\R^3)} \nonumber\\
&= 2\pi\langle Lu,\tilde u\rangle_{L^2_{\rm{cyl}}} \label{eq_quad_form}\\
&= 2\pi\langle u,L\tilde u\rangle_{L^2_{\rm{cyl}}} \mbox{ since $L$ is selfadjoint} \nonumber\\
&= \langle U,\cL \tilde U\rangle_{L^2(\R^3)}. \nonumber
\end{align}
To show that $\cL$ is selfadjoint it suffices to show that for some $\mu\in\R$ the operator 
$$
\cL-\mu\Id: D(\cL)\to L^2_{G_1}(\R^3)
$$
is onto, cf. \cite{Ed_Ev}, Theorem 4.2. We choose any $\mu$ in the resolvent set of $L$, e.g. $\mu=-\|V\|_\infty-1$. Let $F\in L^2_{G_1}(\R^3)$, i.e., there exists $f\in L^2_{\rm{cyl}}(r^3drdx_3)$ with 
$$
F(x)=f(r,x_3) \begin{pmatrix} -x_2 \\ x_1 \\ 0 \end{pmatrix}.
$$
Since $\mu$ lies in the resolvent set of $L$ we can find $u\in D(L)=H^2_{\rm{cyl}}(r^3drdx_3)$ such that 
$Lu-\mu u = f$. Defining 
$$
U(x)=u(r,x_3) \begin{pmatrix} -x_2 \\ x_1 \\ 0 \end{pmatrix}
$$
and using \eqref{beziehung} we get $\cL U-\mu U=F$. This finishes the proof of the selfadjointness of $\cL$ and also of the identity of the resolvent sets of ${\mathcal L}$ and of $L$ and hence $\sigma(\cL)=\sigma(L)$ follows.
\end{proof}

We assume now that our cylindrical waveguide geometry is $1$-periodic along the $x_3$-direction, i.e., $V(r,x_3)=V(r,x_3+1)$ for a.a. $r>0$ and $x_3\in \R$. 
Besides the cylindrical symmetry and the periodicity in $x_3$-direction, the existence of ground states in the focusing case relies on the assumption that $0\not \in \sigma(\mathcal{L})$. In the following we will construct an example of a potential $V=V(r,x_3)$ which is $1$-periodic w.r.t. $x_3$, with $0\not \in \sigma(\mathcal{L})$ and with $\sigma(\mathcal{L})\cap (-\infty,0)\not =\emptyset$. Recall that the physical significance of the sign of $V$ has been explained at the beginning of this section. 

\medskip

Let us assume that the linear potential $V$ is separable, i.e.
\beq\label{E:V_sep}
V(r,x_3)=W(r)+P(x_3).
\eeq
with $W\in L^\infty(0,\infty)$, $P\in L^\infty(\R)$ and $P(x_3+1)=P(x_3)$ for
all $x_3\in \R$ (later we will assume that $P$ is piecewise continuous in order to have \eqref{spectrum_L_p}). The splitting of the potential implies a splitting of the operator $L$ as follows (so far we consider this only on a formal level): if $u(r,x_3)=v(r)w(x_3)$ then
$$
(Lu)(r,x_3) = w(x_3)(L_r v)(r) + v(r)(L_p w)(x_3)
$$
where $L_r, L_p$ are given by the following differential expressions
$$
L_r = -\frac{1}{r^3}\frac{\partial}{\partial r}\left(r^3\frac{\partial}{\partial r}\right)+W(r),  \qquad L_p = - \frac{\partial^2}{\partial x_3^2}+P(x_3).
$$
We can give these differential expressions the meaning of proper selfadjoint operators by specifying their domains of definition properly as in the following lemma.

\begin{lemma} Let $W\in L^\infty(0,\infty), P\in L^\infty(\R)$. The operator $L_p$ defined on $D(L_p)=H^2(\R)\subset L^2(\R)\to L^2(\R)$ is selfadjoint. The operator $L_r$ defined on $D(L_r)=H^2_{\rm{rad}}(r^3dr)\subset L^2_{\rm{rad}}(r^3dr)\to L^2_{\rm{rad}}(r^3dr)$ is selfadjoint.
\end{lemma}
\begin{proof} The statement for $L_p$ is clear. The statement for $L_r$ follows from the observation that for $v\in H^2_{\rm{rad}}(r^3dr)$
$$
(L_r v)\circ \Psi_{\rm{rad}} = -\Delta (v\circ \Psi_{\rm{rad}}) + (Wv)\circ\Psi_{\rm{rad}},
$$
i.e., $L_r$ coincides with the four dimensional Schr\"odinger operator $-\Delta + W(r)$ with radial symmetry. 
\end{proof}

Now we are in a position to state that for $V(r,x_3)=W(r)+P(x_3)$ the spectrum of $\mathcal{L}$ can
be computed by separation of variables. 

\begin{lemma} \label{lem:separabel}
Let $W\in L^\infty(0,\infty), P\in L^\infty(\R)$ and $V(r,x_3)=W(r)+P(x_3)$. Then $\sigma(\cL)=\sigma(L)=\sigma(L_r)+\sigma(L_p)$.
\end{lemma}
\begin{proof} Let us define the subspace 
$$
D_0 = \left\{ \sum_{k=1}^N v_k(r) w_k(x_3): N\in \N, v_k\in D(L_r), w_k\in D(L_p) \mbox{ for } k=1,\ldots,N.\right\}
$$
and the operator $L_r+L_p$ on $D_0$ by 
$$
(L_r+L_p)\left(\sum_{k=1}^N v_k w_k\right):= \sum_{k=1}^N \bigl((L_r v_k)w_k + v_k (L_p w_k)\bigr).
$$
The subspace $D_0$ is dense in $L^2_{\rm{cyl}}(r^3drdx_3)$ because $D(L_r)\subset L^2_{\rm{rad}}(r^3dr)$ and $D(L_p)\subset L^2(\R)$ are dense. Since $L_r+L_p$ is symmetric, it is therefore closable. Let us recall the definition of the closure of an operator and its domain:
\begin{align*}
D(\overline{L_r+L_p})=\bigl\{ & u \in L^2_{\rm{cyl}}(r^3drdx_3): \exists (u_n)_{n\in\N} \mbox{ in } D_0, z\in L^2_{\rm{cyl}}(r^3drdx_3) \mbox{ s.t.} \\
& u_n\to u \mbox{ in } L^2_{\rm{cyl}}(r^3drdx_3) \mbox{ and }  (L_r+L_p)u_n\to z \mbox{ in }  L^2_{\rm{cyl}}(r^3drdx_3) \bigr\}.
\end{align*}
For $u\in D(\overline{L_r+L_p})$ one defines $\overline{L_r+L_p}(u):=z$. By Theorem~VIII.33 and its Corollary from \cite{RS1}, the selfadjointness of $L_r$, $L_p$ is passed on and yields selfadjointness of $\overline{L_r+L_p}$. Note also that $L_r+L_p= L|_{D_0}$ and hence 
$\overline{L_r+L_p}= \overline{L|_{D_0}}$. Since by Lemma~\ref{l:self_l} the operator $L$ defined on $H^2_{\rm{cyl}}(r^3drdx_3)$ is a selfadjoint extension of $L|_{D_0}$ and hence also of the operator $\overline{L|_{D_0}}=\overline{L_r+L_p}$ which is already selfadjoint, we find that 
$\overline{L_r+L_p}=L$. Again by Theorem~VIII.33 and its Corollary from \cite{RS1} we find the claim $\sigma(L)=\sigma(\overline{L_r+L_p})=\overline{\sigma(L_r)+\sigma(L_p)}=\sigma(L_r)+\sigma(L_p)$, where the last equality holds since the two spectra $\sigma(L_r), \sigma(L_p)$ are closed and bounded from below.
\end{proof}

%
%

Next to the periodicity and boundedness of $P$ let us now sharpen the assumption by requiring additionally that $P$ is piecewise continuous. Then the spectrum of $L_p$ is purely continuous and consists of the union of countably many intervals 
\begin{equation}
\sigma(L_p)=\bigcup_{k=1}^\infty [\nu_{2k-1},\nu_{2k}] \mbox{ with } \nu_{2k-1}<\nu_{2k}
\leq \nu_{2k+1},
\label{spectrum_L_p}
\end{equation}
see Theorem XIII.90 in \cite{RS4}. We assume that the first gap is open, i.e., $$\nu_2<\nu_3.$$ 
Next we describe the radial part of the spectrum of $L$ under some special assumptions on the potential $W$. We start with some properties of Bessel functions.

\begin{lemma} \label{l:ex1} Let $J_1$ denote the order one Bessel function which is regular at $0$ and $K_1$ the order one modified Bessel function which decreases exponentially at infinity. Let $0<j_1<j_2<\ldots$ be the positive zeroes of $J_1$ and $0<j_1'<j_2'<\ldots$ be the positive zeroes of $J_1'$. Let 
$$
\eta_\ast = \sqrt{(j_1)^2-(j_1')^2}, \quad \eta^\ast := \sqrt{(j_2')^2-(j_1)^2}.
$$
Then $\eta_\ast<\eta^\ast$ and for every $\eta\in [\eta_\ast,\eta^\ast]$ there exists a unique value $\xi=\xi(\eta)\in (j_1', j_1)$ with the properties
\beq
\frac{J_1(\xi)}{\xi J_1'(\xi)} = \frac{K_1(\eta)}{\eta K_1'(\eta)} \quad \mbox{ and } \quad  (j_1)^2 < \xi^2+\eta^2 < (j_2')^2.
\label{intersect}
\eeq
\label{bessel_zeugs}
\end{lemma}
\begin{proof} 
Define $\tilde g: (0,j_2')\setminus\{j_1'\}\to \R$ and $\tilde h:(0,\infty)\to \R$ by 
$$
\tilde g(\xi) := \frac{J_1(\xi)}{\xi J_1'(\xi)}, \quad \tilde h(\eta) := \frac{K_1(\eta)}{\eta K_1'(\eta)}.
$$
Let us mention that the properties of $\tilde g, \tilde h$ used in this proof are proved in Lemma~\ref{prop_bessel} in the Appendix. Since $\tilde g(\xi)\to 1$ as $\xi\to 0+$ and $\tilde h(\eta)\to -1$ as $\eta\to 0+$, the two functions can be extended continuously to $0$. Moreover, on $[0,j_1')$ the function $\tilde g$ is strictly increasing from $1$ to $+\infty$, on $(j_1',j_2')$ it increases strictly from $-\infty$ to $+\infty$ with a zero at $j_1$. The function $\tilde h$ is negative and strictly increases on $[0,\infty)$ from $-1$ to $0$. Suppose a value $\eta>0$ is given. By the strict monotonicity of $\tilde g$ we can find a unique 
solution $\xi=\xi(\eta)$ of $\tilde g(\xi) = \tilde h(\eta)$ within the interval $(j_1',j_1)$. Now we want to ensure that the pair $(\xi(\eta),\eta)$ satisfies the constraint in \eqref{intersect}. Since $\xi(\eta)\in (j_1',j_1)$ the constraint $(j_1)^2 < \xi^2+\eta^2 < (j_2')^2$ is certainly satisfied if we impose the following restriction on $\eta$:
$$
\eta^2 \in \left((j_1)^2-(j_1')^2,(j_2')^2-(j_1)^2\right) = (\eta_*^2,\eta^{*2}).
$$
Note that $\eta^\ast>\eta_\ast$ is equivalent to $2 (j_1)^2 < (j_1')^2+(j_2')^2$. This inequality can be checked using the numerical values in Table 9.5 of \cite{AbrSteg}. Up to an error in omitted digits the values are $j_1=3.83171$, $j_1'=1.84118$, and $j_2'=5.33144$ so that $2 (j_1)^2 < (j_1')^2+(j_2')^2$ holds. 
\end{proof} 

\begin{lemma} \label{l:ex2} Assume $\nu_2<\nu_3$ and choose values $\mu_0, W_\infty$ such that 
\begin{equation}\label{E:mu0_Winf}
-\nu_3 < \mu_0 < -\nu_2 < -\nu_1 < W_\infty.
\end{equation}
Let $\eta\in [\eta_\ast,\eta^\ast]$ and let $\xi(\eta)$ be as in Lemma~\ref{bessel_zeugs}. If we define 
$$
\delta := \frac{\eta}{\sqrt{W_\infty-\mu_0}} \quad \mbox{ and } \quad W_0 := \mu_0-\left(\frac{\xi(\eta)}{\delta}\right)^2
$$
as well as 
\beq\label{E:V1}
W(r) = \left\{
\begin{array}{ll}
W_0, & 0\leq r < \delta, \vspace{\jot}\\
W_\infty, & r \geq \delta,
\end{array} \right.
\eeq
then $\mu_0$ is an eigenvalue of $L_r$. There are no other eigenvalues below the essential spectrum $[W_\infty,\infty)$ and hence 
$$
\sigma(L_r)=\{\mu_0\}\cup[W_\infty,\infty).
$$
\label{single_ev}
\end{lemma}
\begin{proof} Due to the form of $W$ as in \eqref{E:V1} we have $\sigma(L_r)\subset [W_0,\infty)$ and
$\sigma_{\text{ess}}(L_r)=[W_\infty,\infty)$. Now consider the eigenvalue equation
\beq
\label{eigen_equation}
-u''(r) -\frac{3}{r} u'(r) = \left(-W(r)+\mu\right)u.
\eeq
Let us check that neither $W_0$ nor $W_\infty$ are eigenvalues of $L_r$. First suppose $\mu=W_0$ is an eigenvalue. Note that $-W(r)+W_0\leq 0$ on $[0,\infty)$ and $<0$ on $(\delta,\infty)$. Multiplication of \eqref{eigen_equation} with a corresponding eigenfunction $u$ and integration with respect to the measure $r^3\,dr$ on $(0,\infty)$ yields a positive left hand side and a negative right hand side. Hence $\mu=W_0$ is not an eigenvalue. Now suppose $\mu=W_\infty$ is an eigenvalue and $u$ is a corresponding eigenfunction. Then $-W(r)+W_0=0$ for $r\geq \delta$ so that $u(r)=\const r^{-2}$ on $[\delta,\infty)$. But no matter how $r^{-2}$ extends to $[0,\delta)$ the function $u$ does not belong to $L^2_{\rm{rad}}(r^3dr)$ because $\int_\delta^\infty r^{-4}\cdot r^3dr=\infty$. So $\mu=W_\infty$ is also not an eigenvalue of $L_r$. 

\medskip

Since $\min \sigma(L_r)\geq W_0$ and since neither $W_0$ nor $W_\infty$ are eigenvalues of $L_r$, we are looking for solutions of \eqref{eigen_equation} with $W_0 < \mu < W_\infty$. As $W(r)$ only takes the values $W_0, W_\infty$, equation \eqref{eigen_equation} is transformed via $u(r)=r^{-1}v(\sqrt{\mu-W_0}r), s:=\sqrt{\mu-W_0}r$ into 
$$
s^2 v'' + sv' + (s^2-1)v = 0 \mbox{ for }  0 \leq s \leq \delta \sqrt{\mu-W_0}
$$
and via $u(r) = r^{-1}w(\sqrt{W_\infty-\mu}r), s:=\sqrt{W_\infty-\mu}r$ into
$$
s^2 w'' + sw' -(s^2+1)w = 0 \mbox{ for } \delta\sqrt{W_\infty-\mu} \leq s < \infty.
$$
Thus, $v(s)= \alpha J_1(s)$ is a multiple of the order one Bessel function which is regular at $0$ and $w(s)=\beta K_1(s)$ is a multiple of the order one modified Bessel function which is exponentially decaying at infinity. Altogether we obtain 
$$
u(r) = \left\{ \begin{array}{ll}
\alpha r^{-1}J_1(\sqrt{\mu-W_0}r) & \mbox{ for } 0 \leq r \leq \delta, \vspace{\jot}\\
\beta r^{-1}K_1(\sqrt{W_\infty-\mu}r) & \mbox{ for } \delta \leq r < \infty.
\end{array}\right.
$$
We need to choose $\alpha,\beta, \mu$ in order to obtain a $C^1$-function at $r=\delta$. This leads to the equation 
$$
g(\mu):=\frac{J_1(\sqrt{\mu-W_0}\delta)}{\sqrt{\mu-W_0}\delta J_1'(\sqrt{\mu-W_0}\delta)} = \frac{K_1(\sqrt{W_\infty-\mu}\delta)}{\sqrt{W_\infty-\mu}\delta K_1'(\sqrt{W_\infty-\mu}\delta)}=:h(\mu)
$$
and our choice of $W_0$ and $\delta$ such that $\sqrt{\mu_0-W_0}\delta=\xi$ and $\sqrt{W_\infty-\mu_0}\delta=\eta(\xi)$ guarantees the $C^1$-matching at $\mu=\mu_0$. We have therefore verified that $\mu_0$ is indeed an eigenvalue of $L_r$. 
It remains to show that there is no other eigenvalue. 

We analyze the two sides of the equation $g(\mu)=h(\mu)$ independently. We have already mentioned in the proof of the preceeding lemma that  
$\frac{K_1(x)}{xK_1'(x)}$ is a negative and increasing function of $x$ with $\lim_{x\to 0+} \frac{K_1(x)}{xK_1'(x)}=-1$. Likewise the function $\frac{J_1(x)}{xJ_1'(x)}$  satisfies $\lim_{x\to 0+}\frac{J_1(x)}{xJ_1'(x)}=1$, has its zeroes at $j_1, j_2, \ldots$ and its poles at $j_1', j_2',\ldots$ and is increasing between $0$ and $j_1'$ and between two consecutive poles. The proof of these statements is given in the Appendix. Thus, as $\mu$ runs through $[W_0,W_\infty]$, the function $g(\mu)$ starts from the value $1$ and increases up to its first pole at $W_0+\left(\frac{j_1'}{\delta}\right)^2$. On the interval $(W_0 +\left(\frac{j_1'}{\delta}\right)^2,W_0+\left(\frac{j_1}{\delta}\right)^2]$ it increases from $-\infty$ to $0$ and on $[W_0+\left(\frac{j_1}{\delta}\right)^2, W_\infty]$ it increases from $0$ to a positive value. The function $h(\mu)$ stays negative and strictly decreases to the value $-1$ as $\mu$ ranges through the interval $[W_0,W_\infty]$, cf. Figure~\ref{sketch} for a plot of an example of the two functions. Therefore, on $[W_0,W_\infty]$ the two functions $g, 
h$ 
intersect exactly once provided $W_\infty$ lies between the first zero and the second pole of $g$, i.e., provided 
\begin{equation}\label{E:ev-cond}
\left(\frac{j_1}{\delta}\right)^2 < W_\infty -W_0 < \left(\frac{j_2'}{\delta}\right)^2. 
\end{equation}\eqref{E:ev-cond} is equivalent to $j_1^2<\xi^2+\eta^2<(j_2')^2$. The latter is guaranteed by our choice of $\xi, \eta$ and from Lemma~\ref{bessel_zeugs}.
\begin{figure}\label{sketch}
\scalebox{0.55}{\includegraphics{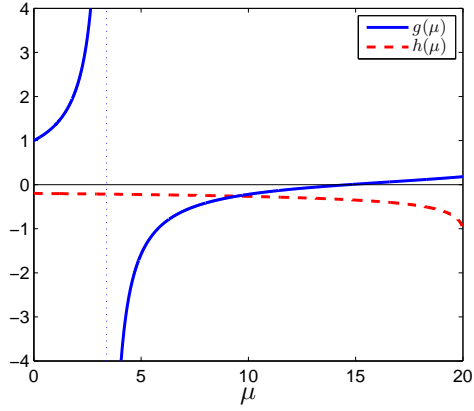}}
\caption{Functions $g$ and $h$ for $\delta=1, W_0=0, W_\infty=20$.}
\end{figure}
\end{proof}

As we have seen from Lemma~\ref{single_ev}, the piecewise constant function $W$ given in \eqref{E:V1} together with the choice of $\mu_0, W_\infty$, the restriction on $\eta$ and the definition of $\delta$ and $W_0$ ensures the existence of exactly one eigenvalue of $L_r$ below the essential spectrum. Therefore these conditions can be considered as assumptions that ensure guiding of a single linear mode in the cylindrical waveguide.

\begin{lemma} \label{l:ex3}
Assume that $P$ is piecewise continuous, $1$-periodic such that $\sigma(L_p)$ has a bounded first open gap. Assume that $W$ is as in \eqref{E:V1} and $\mu_0, W_\infty, \delta$ and $W_0$ are chosen as in Lemma~\ref{single_ev}. Then $0$ is not in the spectrum of $L$.  
\end{lemma}

\begin{proof} Because of Lemma~\ref{lem:separabel}, Lemma~\ref{l:ex2} and \eqref{spectrum_L_p} we see that 
$$
\sigma(L) = \bigcup_{k=1}^\infty [\mu_0+\nu_{2k-1},\mu_0+\nu_{2k}] \cup
[\nu_1+W_\infty,\infty).
$$
Thus, the inequalities $-\nu_3 < \mu_0 < -\nu_2 < -\nu_1 < W_\infty$ from Lemma~\ref{single_ev} guarantee that $0$ lies in the resolvent of $L$. 
\end{proof}

\noindent
{\bf Remark.} Let us verify that the constructed potential $V$ takes positive values on sets of positive measure for $r>\delta$. Since $V(x)= -\frac{\omega^2}{c^2}n^2(x)$, this implies the unphysical situation of an imaginary refractive index. It is an open problem to construct a negative $V=V(r,x_3)$, which satisfies the condition $0\not\in \sigma(\cL)$. First note that 
\begin{align*}
\nu_1 = \min \sigma(L_p) &= \inf_{\psi\in H^1(\R)\setminus\{0\}} \frac{\int_\R {\psi'}^2+P(x)\psi^2\,dx}{\int_\R \psi^2\,dx} \\ 
& < \inf_{\psi\in H^1(\R)\setminus\{0\}} \frac{\int_\R {\psi'}^2+(\esssup P) \psi^2\,dx}{\int_\R \psi^2\,dx}\\
&= \esssup P,
\end{align*}
where the strict inequality comes from the fact that $P$ is a periodic potential which generates an operator with a true first open gap and hence must be non-constant. Together with $W_\infty > -\nu_1$, cf. \eqref{E:mu0_Winf} we obtain
$$
\esssup V = W_\infty+ \esssup P >0.
$$

\section{Ground states in the focusing case}\label{S:foc}

We assume 
$$
1<p<5, \quad  V,\Gamma \in L^\infty(\R^3),  \quad \essinf_{\R^3} \Gamma>0,  \quad \text{ and } \quad 0\notin \sigma(\cL). \leqno {\mbox{(H-foc)}}
$$
For the existence of a Palais-Smale sequence our only assumptions is (H-foc). In the previous section we found examples of cylindrical potentials $V$ of the form $V(r,x_3)=W(r)+P(x_3)$ with periodicity of $P$ which produce $0 \not \in \sigma(\mathcal{L})$. Nevertheless, also other potentials $V$ may satisfy this condition. Note that for the potential $V$ from the previous section the resulting operator $\mathcal{L}$ has negative spectrum, positive spectrum and $0$ is in the resolvent set. The analysis of this chapter will, however, work also in the case where $0<\min\sigma(\mathcal{L})$.

\medskip

For the final step of the proof, i.e. the non-triviality of the limit of a suitable Palais-Smale sequence and the existence of a nontrivial ground state, we use the assumption of cylindrical symmetry and periodicity in $x_3$ of both $V$ and $\Gamma$. 

\medskip

Since ${\mathcal L}$ is a selfadjoint operator on 
$$
D(\mathcal{L})= H_{G_1}^2(\R^3)\subset L^2_{G_1}(\R^3)
$$ 
the spectral theorem yields the existence of a spectral resolution $(P_\lambda)_{\lambda\in\R}$ and hence of projections 
$$
P^+ = \int_0^\infty 1 \,d P_\lambda, \quad P^- = \int_{-\infty}^0 1 \,d P_\lambda
$$
and of the spectral decomposition $H_{G_1}^1(\R^3)= H^+\oplus H^-$ wit $H^\pm=P^\pm(H^1_{G_1}(\R^3))$. We define $U^\pm = P^\pm U$ for all $U\in H_{G_1}^1(\R^3)$. Notice that $\mathcal{L}$ defines the bilinear form $b(U,V) := \int_{\R^3} (\nabla \times U )\cdot (\nabla\times V) + V(x) U\cdot V\,dx$ which is positive/negative definite on the spaces $H^\pm$ by construction. Therefore, we may define the scalar product
$$
\langle U,V\rangle := b(U^+,V^+)-b(U^-,V^-) \quad \mbox{ on } H^1_{G_1}(\R^3)\times H^1_{G_1}(\R^3)
$$
which has the property that $H^+$ is orthogonal to $H^-$, i.e., that $P^\pm$ are orthogonal projections. We denote the norm on $H_{G_1}^1(\R^3)$ corresponding to $\langle \cdot,\cdot\rangle$ by $\interleave \cdot \interleave$. It is equivalent to the $H^1$-norm $\|\cdot\|_{H^1}$ and has the property
$$
\interleave U\interleave = \pm \int_{\R^3} |\nabla \times U|^2+ V(x) |U|^2\,dx \mbox{ for all } U\in H^\pm
$$
such that
$$
\interleave U\interleave^2 = \interleave U^+\interleave^2 +\interleave U^-\interleave^2 \mbox{ for all } U\in H^1_{G_1}(\R^3). 
$$
If the potential $V$ has the unphysical property that $\essinf V>0$ (imaginary refractive index in all of $\R^3$) then $H^-=\{0\}$. But in general $H^-\not = \{0\}$ and therefore the functional $J$ has linking geometry. In any case $J$ is unbounded from below on $H_{G_1}^1(\R^3)$ so that a direct minimization is impossible. Therefore we choose to minimize over the Nehari-Pankov manifold
\[N:=\{U\in H^1_{G_1}(\R^3)\setminus \{0\}: J'[U]\Phi=0 \ \forall \Phi\in
[U]\oplus H^-\}.\]
This approach is analogous to that in \cite{Pankov05}. Later, in Lemma~\ref{natural_constraint}, we will see that the constraint set $N$ does not produce Lagrange multipliers. Note also that for $U\in N$ 
$$
J[U]= \frac{p-1}{2(p+1)} \int_{\R^3} \Gamma(x) |U|^{p+1}\,dx = \frac{p-1}{2(p+1)}\int_{\R^3} |\nabla \times U|^2+ V(x) |U|^2\,dx.
$$

\noindent
{\bf Remarks:} 
(a) Another common approach to obtain a critical point of $J$ is minimization under the
constraint $\|U\|_{L^2}=1$. This however produces a Lagrange multiplier
$\kappa$, which is generally nonzero, so that the minimizer does not solve
\eqref{E:curl-curl-gen} but the equation $\mathcal{L}U+\kappa
U=\Gamma(r,x_3)|U|^{p-1}U$.\\
(b) If $H^-=\{0\}$ then another common approach consists in minimizing $K(U)=\int_{\R^3}|\nabla\times U|^2+V(x)|U|^2\,dx$ under the constraint $\|U\|_{\Gamma,p+1}=1$. This produces a Lagrange multiplier which however can be scaled out. If $H^-$ is not trivial then one may still find a critical point of $K$ under this constraint.

\begin{lemma} \label{bounds}
Under the assumption {\rm (H-foc)} there exist values $\epsilon_1, \epsilon_2, C>0$ such that
$$
\|U\|_{H^1}\geq \epsilon_1, \quad J[U] \geq \epsilon_2, \quad \|U\|_{H^1} \leq
CJ[U]^\frac{p}{p+1}
$$
for all $U\in N$.
\end{lemma}

\begin{proof} For $U\in N$ we have 
\begin{align*}
\interleave U^+\interleave^2 &= \int_{\R^3} |\nabla\times U^+|^2+ V(x) |U^+|^2 \,dx \\
&= \langle U, U^+ \rangle \\
&= J'[U] U^+ + \int_{\R^3} \Gamma(x) |U|^{p-1} U\cdot U^+\,dx \\
&= \underbrace{J'[U](U-U^-)}_{=0} + \int_{\R^3} \Gamma(x) |U|^{p-1} U\cdot U^+\,dx,
\end{align*}
where the first term in the last equation vanishes due to the definition of the manifold $N$. Hence, by using H\"older's and Sobolev's inequality and again the definition of $N$ we obtain
\begin{align}
\interleave U^+\interleave^2& \leq \left(\int_{\R^3} (\Gamma(x) |U|^p)^\frac{p+1}{p}\,dx\right)^\frac{p}{p+1} \left(\int_{\R^3} |U^+|^{p+1}\,dx\right)^\frac{1}{p+1} \nonumber \\
& \leq C_1 \|\Gamma\|_\infty^\frac{1}{p+1} \left(\int_{\R^3} \Gamma(x) |U|^{p+1}\,dx\right)^\frac{p}{p+1} \|U^+\|_{H^1} \label{norm_estimate} \\
& \leq C_2\left(\frac{2(p+1)}{p-1}J[U]\right)^\frac{p}{p+1} \interleave U^+\interleave. \nonumber
\end{align}
We may repeat the above argument with $U^-$ and find 
\begin{align*}
\interleave U^-\interleave^2 &= -\int_{\R^3} |\nabla\times U^-|^2+ V(x) |U^-|^2 \,dx\\
&= \langle U, U^- \rangle \\
&= -\underbrace{J'[U] U^-}_{=0} -\int_{\R^3} \Gamma(x) |U|^{p-1} U\cdot U^-\,dx,
\end{align*}
and from there the same H\"older- and Sobolev-estimates as before lead to 
$$
\interleave U^-\interleave \leq  C_2\left(\frac{2(p+1)}{p-1}J[U]\right)^\frac{p}{p+1}.
$$
Together with \eqref{norm_estimate} this establishes the third of the three claims. 

\medskip

To see the first of the three claims, we use \eqref{norm_estimate} and the corresponding estimate for $U^-$ to get 
$$
\interleave U^+ \interleave^2 \leq C_3 \|U\|_{L^{p+1}}^p \interleave U^+\interleave, \quad \interleave U^- \interleave^2 \leq C_3 \|U\|_{L^{p+1}}^p \interleave U^-\interleave
$$
from which we obtain $\interleave U\interleave \leq \sqrt{2}C_3 \|U\|_{L^{p+1}}^p$. Due to the Sobolev inequality and since $\interleave\cdot\interleave$ and $\|\cdot\|_{H^1}$ are equivalent we also have $\|U\|_{L^{p+1}}^p\leq C_4 \interleave U\interleave^p$. Hence $ \interleave U\interleave \leq \sqrt{2}C_3C_4 \interleave U\interleave ^p$ and thus $\|U\|_{H^1}\leq C_5 \|U\|^p$. Since $U\not =0$ by the definition of $N$ we obtain the first of the three estimates. Finally, the second estimate follows from the first and the third.
\end{proof}

\begin{lemma} \label{linearization}
Assume {\rm (H-foc)}. The map 
$$
G: \left\{ \begin{array}{rcl} 
H_{G_1}^1(\R^3)\setminus H^- & \to & H_{G_1}^1(\R^3), \vspace{\jot} \\
U & \mapsto & \langle \nabla J[U],U^+\rangle \frac{U^+}{\interleave U^+\interleave^2} + P^- \nabla J[U].
\end{array} \right.
$$
is a $C^1$-map. If $U \in N$ and $X_U:=[U]+H^-$ then the map 
$$
\partial_{X_U} G(U): X_U \to X_U
$$
is negative definite uniformly with respect to bounded subsets of $N$, i.e., if $N_0\subset N$ is bounded, then there exists a value $\delta>0$ such that 
$$
\langle \partial_{X_U} G(U) v,v\rangle \leq -\delta \interleave v\interleave^2 \mbox{ for all } v\in X_U \mbox{ and all } U\in N_0.
$$
In particular $\partial_{X_U} G(U): X_U\to X_U$ has a bounded inverse.
\end{lemma}

\begin{proof} For $U\in N$ and $v \in H_{G_1}^1(\R^3)$ we have 
\begin{align*}
G'(U)v =& \langle \nabla^2 J[U]v, U^+\rangle\frac{U^+}{\interleave U^+\interleave^2} + \langle \nabla J[U],v^+\rangle \frac{U^+}{\interleave U^+\interleave^2} + \langle \nabla J[U], U^+\rangle \frac{v^+}{\interleave U^+\interleave^2} \\
& - 2 \langle U^+,v\rangle \langle \nabla J[U],U^+\rangle \frac{U^+}{\interleave U^+\interleave^4} + (\nabla^2 J[U]v)^-.
\end{align*}
If we take $v \in X_U=[U]+H^- = [U^+]+ H^-$, i.e., $v = tU+w$ with $t\in \R$, $w\in H^-$ then the above formula simplifies to 
$$
\partial_{X_U} G(U)v =\langle \nabla^2 J[U]v, U^+\rangle\frac{U^+}{\interleave U^+\interleave^2}+(\nabla^2 J[U]v)^- =P_{X_U} \nabla^2 J[U]v,
$$
where $P_{X_U}: H\to X_U$ denotes the orthogonal projection with respect to $\langle\cdot,\cdot\rangle$. Therefore 
\begin{align*}
\langle \partial_{X_U} G(U) v,v\rangle =& \langle \nabla^2 J[U]v,v\rangle \\
=& t^2 \int_{\R^3} |\nabla \times U|^2 + V(x) |U|^2- p\Gamma(x) |U|^{p+1}\,dx \\
& +\int_{\R^3} |\nabla \times w|^2 + V(x) |w|^2- p\Gamma(x) |U|^{p-1} |w|^2\,dx \\
& + 2t \int_{R^3} \nabla \times U \cdot \nabla\times w + V(x) U\cdot w- p \Gamma(x) |U|^{p-1} U\cdot w\,dx
\end{align*}
and by using $U\in N$ we obtain
$$
\langle \partial_{X_U} G(U) v,v\rangle = -\interleave w \interleave^2 - \int_{\R^3} \Gamma(x) |U|^{p-1}\left(t^2(p-1) |U|^2 + p |w|^2+2t(p-1)U\cdot w\right) \,dx.
$$
Now we use the identity 
\begin{align*}
t^2(p-1) |U|^2 + p |w|^2+2t(p-1)U\cdot w &= t^2\frac{p-1}{p} |U|^2+ \left| \sqrt{p} w + \frac{p-1}{\sqrt{p}}t U \right|^2 \\
& \geq t^2\frac{p-1}{p} |U|^2
\end{align*}
to deduce
$$
\langle \partial_{X_U} G(U) v,v\rangle \leq -\interleave w \interleave^2- t^2\frac{p-1}{p}\int_{\R^3} \Gamma(x)|U|^{p+1}\,dx.
$$
Next we obtain from Lemma~\ref{bounds} that $J[U]= (\frac{1}{2}-\frac{1}{p+1})\int_{\R^3} \Gamma(x)|U|^{p+1}\,dx \geq \epsilon_2$ and hence we find 
$$
\langle \partial_{X_U} G(U) v,v\rangle \leq -\interleave w \interleave^2- 2t^2\frac{p+1}{p}\epsilon_2 \leq -C\interleave w+tU\interleave^2
$$
by using the boundedness of $N_0$. This finishes the proof.
\end{proof}

\begin{lemma} \label{manifold}
Assume {\rm (H-foc)}. The set $N$ is a $C^1$-manifold such that 
$$
N = G^{-1}\{0\} \mbox{ and } T_{U} N = \kernel G'(U)
$$
for every $U\in N$. 
\end{lemma}

\begin{proof} As before, for $U\in N$ let $X_U= [U]+H^-$ and denote by $P_{X_U}: H^1_{G_1}(\R^3) \to X_U$ the orthogonal projection w.r.t. $\langle\cdot,\cdot\rangle$. Recall that the map $G:H^1_{G_1}(\R^3)\setminus H^-\to H^1_{G_1}(\R^3)$ satisfies 
\begin{equation}
G(X_U\setminus H^-) \subset X_U \mbox{ for all } U\in H^1_{G_1}(\R^3)\setminus H^-.
\label{property}
\end{equation}
Now fix $U\in N=G^{-1}(\{0\})$. There exists an open neighbourhood $\mathcal{O}$ of $U$ in $H^1_{G_1}(\R^3)\setminus H^-$ such that 
\begin{equation}
X_V \cap X_U^\perp = \{0\} \mbox{ for all } V\in \mathcal{O}.
\label{property2}
\end{equation}
For $V\in \mathcal{O}$ we have by \eqref{property}, \eqref{property2} the equivalence 
$$
G(V)=0 \Leftrightarrow (P_{X_U}\circ G)(V)=0.
$$
This shows that $N\cap \mathcal{O} = (P_{X_U}\circ G\mid_{\mathcal{O}})^{-1}(\{0\})$. But the map 
$P_{X_U}\circ G\mid_{\mathcal{O}}: \mathcal{O}\to X_U$ is a submersion, i.e., its derivative is surjective at every point of $\mathcal{O}$, because 
$$
(P_{X_U}\circ G)'(V)= P_{X_U} \circ G'(V): H^1_{G_1}(\R^3) \stackrel{G'(V)}{\to} X_V \stackrel{P_{X_U}}{\to} X_U
$$
and the first map $G'(V): H^1_{G_1}(\R^3)\to X_V$ is surjective by Lemma~\ref{linearization} and the second map $P_{X_U}: X_V \to X_U$ is an isomorphism by \eqref{property2}. Therefore, the submersion theorem of \cite{amr}, Theorem 3.5.4, applies and states that $N\cap\mathcal{O}$ is a submanifold of $H^1_{G_1}(\R^3)\setminus H^-$ and $T_U N = \kernel G'(U)$. 
\end{proof}

Notice that any nontrivial solution $U\in H^1_G(\R^3)$
of  \eqref{E:curl-curl-gen} belongs to the Nehari-Pankov manifold $N$. As a
consequence, one can show that the constraint $N$ produces a zero Lagrange
multiplier. The following is a much stronger statement.

\begin{lemma} \label{natural_constraint}
Assume {\rm (H-foc)}. Let $N_0$ be a bounded subset of $N$. Then there exists a constant $C_0>0$ such that the following holds: if $U \in N_0$ and $\nabla J[U]=\tau+\sigma$ where $\tau \in T_{U} N$ is the tangential component of $\nabla J[U]$ and $\sigma\perp\tau$ is the transversal component of $\nabla J[U]$ then 
$$
\|\nabla J[U]\|_{H^1} \leq C_0 \|\tau\|_{H^1}.
$$
\end{lemma}

\begin{proof} By Lemma~\ref{linearization} the map $\partial_{X_U} G(U):X_U\to X_U$ has a bounded inverse and hence a closed range. Moreover, $\partial_{X_U} G(U)|_{X_U}= P_{X_U} \nabla^2 J[U]$ is symmetric as a composition of a second derivative and an orthogonal projection. Therefore 
$$
\range \partial_{X_U} G(U)|_{X_U} =\left(\kernel \partial_{X_U} G(U)|_{X_U} \right)^\perp = (T_{U}N)^\perp.
$$
If we consider 
$$
\nabla J[U] = \tau+\sigma \mbox{ with } \tau \in T_{U}N \mbox{ and } \sigma\in (T_{U}N)^\perp
$$
then there exists $h\in X_U$ such that 
\beq
\label{split}
\nabla J[U] = \tau +\partial_{X_U} G(U) h.
\eeq
Hence, using $h\in X_U=[U]+H^-$ and thus $\langle \nabla J[U], h\rangle=0$ we get from Lemma~\ref{linearization}   
$$
\langle \partial_{X_U} G(U) h,h\rangle = \langle -\tau, h\rangle \leq -\delta \|h\|_{H^1}^2.
$$
Using the Cauchy-Schwarz inequality and the equivalence of the norms $\|\cdot\|_{H^1}$ and $\interleave \cdot \interleave$ we get 
$$
\|h\|_{H^1} \leq C_0 \|\tau\|_{H^1}.
$$
By \eqref{split} and the boundedness of $G'(U)$ on bounded subsets of $N$ this implies the claim.
\end{proof}

By the previous lemma the tangential component of the gradient of $J$ at a point in $N$ controls the entire gradient. As a consequence there exists special minimizing sequences of $J|_N$, where the tangential part of the gradient converges to zero (a consequence of Ekeland's variational principle) and hence the entire gradient converges to $0$.

\begin{lemma} \label{ps}
Assume {\rm (H-foc)}. There exists a bounded Palais-Smale sequence
$(U_k)_{k\in \N}$ in $N$ such that 
$$
J[U_k]\to c:=\inf_N J, \quad J'[U_k]\to 0 \mbox{ as } k \to \infty.
$$
\end{lemma}

\begin{proof} As a consequence of Ekeland's variational principle, cf. Struwe~\cite{struwe}, there exists a minimizing sequence $(U_k)_{k\in\N}$ of $J|_N$ such that $ (J|_N)'(U_k) \to 0$, hence $J'(U_k) \to 0$ 
by Lemma~\ref{natural_constraint}.
\end{proof}

\noindent
{\em Proof of Theorem~\ref{foc}:} Let $(U_k)_{k\in \N}$ be the Palais-Smale
sequence from Lemma~\ref{ps}. If any subsequence of $(U_k)_{k\in \N}$ converges
in $L^{p+1}(\R^3)$ to zero, then by the definition of the Nehari-Pankov manifold this
sequence also converges to zero in $H^1(\R^3)$. This is impossible by
Lemma~\ref{bounds}. Therefore, the concentration compactness principle (cf. Lions~\cite{lions} or Lemma~1.21 in Willem~\cite{willem} suitably adapted to the vectorial case) implies that for every radius
$R>0$  
$$
\liminf_{k\in \N} \sup_{y\in \R^3} \int_{B_R(y)} |U_k|^2\,dx >0.
$$
Fix a value $R>0$. Then there exists a subsequence again denoted $(U_k)_{k\in\N}$,
centers $y_k\in \R^3$ and $\eta>0$ such that 
\beq
\int_{B_R(y_k)} |U_k|^2\,dx \geq \eta \mbox{ for all } k\in \N.
\label{integrals}
\eeq
By possibly adding the value 1 to $R$ we may assume that $y_k^3 \in \Z$ and \eqref{integrals} still holds.
Now we claim that $\rho_k^2 := (y_k^1)^2+(y_k^2)^2$ is bounded. Assume the
contrary. Due to the symmetries in $H_{G_1}^1(\R^3)$ we have that 
$$
\int_{B_R(y_k)} |U_k|^2\,dx = \int_{B_R(\tilde y_k)} |U_k|^2\,dx 
$$
whenever the point $\tilde y_k$ is such that 
$$
y_k^3=\tilde y_k^3 \mbox{ and } (y_k^1)^2+(y_k^2)^2=(\tilde y_k^1)^2+(\tilde
y_k^2)^2. \leqno(P)
$$
Notice that the number of disjoint balls $B_R(\tilde y_k)$ with centers $\tilde
y_k$ satisfying $(P)$ tends to infinity if $\rho_k\to \infty$ as $k\to \infty$. But this is
impossible since the $L^2$-norm of $(U_k)_{k\in \N}$ is bounded. Thus, if we
define 
$$
\rho := R+\sup_{k\in \N} \rho_k
$$
then
$$
\int_{B_\rho(0,0,y_k^3)} |U_k|^2\,dx \geq \eta \mbox{ for all } k \in \N.
$$
Set 
$$
\bar U_k(x_1,x_2,x_3) := U_k(x_1,x_2,x_3+y_k^3).
$$
Then, due to the periodicity of $V, \Gamma$ in the $x_3$-variable we have
$$
\bar U_k \in N, \quad J[\bar U_k]= J[U_k]\to c,\quad \|J'[\bar U_k]\|_\ast=\|J'[U_k]\|_\ast \to 0 \mbox{
as } k \to \infty
$$
and 
$$
\int_{B_\rho(0)} |\bar U_k|^2\,dx \geq \eta \mbox{ for all } k \in \N.
$$
Now we may take a weakly converging subsequence (again denoted by $(\bar
U_k)_{k\in\N}$) with $\bar U_k \rightharpoonup \bar U\not =0$ in $H_{G_1}^1(\R^3)$
for some $\bar U \in H_{G_1}^1(\R^3)$. Moreover, for every $\phi\in H^1_{G_1}(\R^3)\cap C_0^\infty(\R^3)$ we have 
$J'[\bar U]\phi = \lim_{k\to \infty} J'[\bar U_k]\phi=0$ due to weak convergence and the compact embedding $H^1_{G_1}(\R^3)\to L^q(K)$, $K=\supp\,\phi$ and $q\in [1,6)$. Hence $\bar U$ is a critical point of $J$ so that $\bar U\in N$. Hence 
\begin{align*}
J[\bar U] &= \left(\frac{1}{2}-\frac{1}{p+1}\right) \int_{\R^3} \Gamma(x) |\bar
U|^{p+1}\,dx \\
& \leq \liminf_{k\in \N}\left(\frac{1}{2}-\frac{1}{p+1}\right) \int_{\R^3}
\Gamma(x) |\bar U_k|^{p+1}\,dx \\
& = \liminf_{k\in \N} J[\bar U_k] \\
& =c
\end{align*}
and therefore ``$=$'' holds  and $\bar U$ is a minimizer of $J$ restricted to
$N$. \qed

\section*{Appendix}

\begin{lemma} Let $J_1$ denote the order one Bessel function which is regular at $0$. Then the function
$$
\alpha(x):=\frac{J_1(x)}{xJ_1'(x)} 
$$
satisfies $\lim_{x\to 0+}\alpha(x)=1$ and is strictly increasing on the intervals $(0, j_1')$, $(j_{k}', j_{k+1}')$ for $k=1,2,3,\ldots$. Let $K_1$ denote the order one modified Bessel function which decreases exponentially at infinity. Then the function
$$
\beta(x):=\frac{K_1(x)}{xK_1'(x)}
$$
satisfies $\lim_{x\to 0+} \beta(x)=-1$ and is negative and strictly increasing on $(0,\infty)$.
\label{prop_bessel}
\end{lemma}

\begin{proof} Since $J_1$ is analytic and since $J_1(0)=0, J_1'(0)=1/2$ the relation $\alpha(x)\to 1$ as $x\to 0$ follows immeadiately. Differentiating $\alpha(x)$, we need to show
$$
\frac{x (J_1'(x))^2-J_1(x)(J_1'(x)+x J_1''(x))}{x^2 (J_1'(x))^2}>0
$$
and by using the differential equation for $J_1$ this amounts to
$$
x\left((J_1'(x))^2 + \left(1-\frac{1}{x^2}\right)J_1(x)^2\right)>0,
$$
i.e.,
$$
x^2(J_1'(x))^2 > (1-x^2) J_1(x)^2 \mbox{ for } x \in (0,\infty).
$$
If we multiply the differential equation
$$
x^2 J_1''(x) +x J_1'(x) = (1-x^2)J_1(x)
$$
by $J_1'(x)$ and integrate from $0$ to $x$ we obtain
$$
\int_0^x s^2 \frac{d}{ds}(J_1'(s))^2+2s (J_1'(s))^2\,ds = \int_0^x (1-s^2)\frac{d}{ds} (J_1(s)^2)\,ds.
$$
Integration by parts and using $J_1(0)=0$ leads to 
$$
x^2 (J_1'(x))^2 = (1-x^2)J_1(x)^2 + \int_0^x 2s J_1(s)^2\,ds > (1-x^2)J_1(x)^2 \mbox{ for } x \in (0,\infty)
$$
and hence the result is proved.

\medskip

Now we turn to the statement for $\beta$. First we recall that $x K_0(x)\to 0$ (cf. \cite{gradshteyn_ryzhik}, 8.447), $K_1(x)\to \infty$ as $x\to 0$ (cf. \cite{gradshteyn_ryzhik}, 8.451(6.)) and that $x K_1'(x)+K_1(x)=-x K_0(x)$ (cf. \cite{gradshteyn_ryzhik}, 8.486(12.)). This implies  $\beta(x)\to -1$ as $x\to 0$. For the monotoniticy of $\beta(x)$ it suffices by differentiation to prove that
$$
x(K_1'(x))^2 - K_1(x) K_1'(x) - x K_1(x)K_1''(x) >0
$$
and using the differential equation this amounts to showing that 
$$
x^2 (K_1'(x))^2> (1+x^2) K_1(x)^2 \mbox{ for all } x\in (0,\infty).
$$
If we multiply the differential equation
$$
x^2 K_1''(x) +x K_1'(x) = (1+x^2)K_1(x)
$$
by $K_1'(x)$ and integrate from $x$ to $\infty$ we obtain
$$
\int_x^\infty s^2 \frac{d}{ds}(K_1'(s))^2+2s (K_1'(s))^2\,ds = \int_x^\infty (1+s^2)\frac{d}{ds} (K_1(s)^2)\,ds.
$$
Integration by parts and using the exponential decay of $K_1, K_1', K_1''$ at infinity leads to 
$$
-x^2 (K_1'(x))^2 = -(1+x^2)K_1(x)^2 -\int_x^\infty 2s K_1(s)^2\,ds < -(1+x^2)K_1(x)^2 \mbox{ for } x \in (0,\infty).
$$
This proves the result.
\end{proof}

\bibliographystyle{abbrv}	
\bibliography{bibliography}

\end{document}